\documentclass[11pt]{article} 

\usepackage[margin=2.5cm]{geometry}
\usepackage{amsmath}
\numberwithin{equation}{section}
\usepackage{tikz}
\usepackage{tikz-cd}
\usepackage{amssymb}
\usepackage{amsthm}
\usepackage{mathtools}
\usepackage{bm}
\usepackage{relsize}
\usepackage{accents}
\usepackage{scalerel,stackengine}
\usepackage{multirow}
\usepackage{subfigure}

\DeclareMathOperator{\ddiv}{div}

\stackMath
\newcommand\reallywidehat[1]{%
\savestack{\tmpbox}{\stretchto{%
  \scaleto{%
    \scalerel*[\widthof{\ensuremath{#1}}]{\kern-.6pt\bigwedge\kern-.6pt}%
    {\rule[-\textheight/2]{1ex}{\textheight}}
  }{\textheight}%
}{0.5ex}}%
\stackon[1pt]{#1}{\tmpbox}%
}

\newcommand{\ed}{\ensuremath{\mathrm{d}}} 
\newcommand{\mc}[1]{\ensuremath{\mathcal{#1}}}
\renewcommand{\bf}[1]{\ensuremath{\mathbf{#1}}}
\renewcommand{\sf}[1]{\ensuremath{\mathsf{#1}}}
\newcommand{\mr}[1]{\ensuremath{\mathrm{#1}}}
\newcommand{\mf}[1]{\ensuremath{\mathfrak{#1}}}

\newcommand{\parder}[2]{\ensuremath{\frac{\partial #1}{\partial #2}}}
\newcommand{\der}[2]{\ensuremath{\frac{\ed #1}{\ed #2}}}

\newcommand{\wh}[0]{ \accentset{\circ}{W}_{h}}
\newcommand{\Wh}[0]{ \accentset{\circ}{\bf W}_{h}}

\usepackage{amssymb}

\let\emptyset\varnothing

\newtheorem{theorem}{Theorem}[section]
\newtheorem{remark}[theorem]{Remark}

\newtheorem{definition}[theorem]{Definition}
\newtheorem{proposition}[theorem]{Proposition}
\newtheorem{lemma}[theorem]{Lemma}

\usepackage{authblk}

\title{A variational ${\bm H}(\ddiv)$ finite element discretisation approach for perfect incompressible fluids}
\author{Andrea Natale and Colin J Cotter}
\affil{Department of Mathematics, Imperial College London, London, SW7 2AZ, UK}
\date{\today}

\begin{document}

\maketitle

\begin{abstract} 
We propose a finite element discretisation approach for the
incompressible Euler equations which mimics their geometric structure
and their variational derivation. In particular, we derive a finite
element method that arises from a nonholonomic variational principle
and an appropriately defined Lagrangian, where finite element
${\bm H}(\ddiv)$ vector fields are identified with advection operators; 
this is the first successful extension of the structure-preserving discretisation of Pavlov et al.\ (2009) to the finite element setting.
The resulting algorithm coincides with the energy-conserving scheme presented in Guzm\'an et al.\ (2016). Through the variational derivation, we discover that it also satisfies a discrete analogous of Kelvin's circulation theorem. 
Further, we propose an upwind-stabilised version of the scheme which
dissipates enstrophy whilst preserving energy
conservation and the discrete Kelvin's theorem. We prove error estimates for this version of the scheme, and we study its behaviour through numerical tests. 
\end{abstract}

\section{Introduction}

The equations of motion for perfect fluids with constant density, namely the incompressible Euler equations, possess a rich geometric structure which is directly responsible for their many properties. More precisely, one can derive these equations from a variational principle using geometric techniques that also apply to classical mechanical systems such as the rigid body, for example.
This geometric view of perfect fluids is due to Arnol'd \cite{Arnold66}. It has been extended to a wide variety of fluid models including  Magnetohydrodynamics (MHD) and Geophysical Fluid Dynamics (GFD) models through the incorporation of advected quantities \cite{Holm98}. 
In this paper, we derive finite element numerical methods for perfect incompressible fluids, aiming to preserve as much of this geometric structure as possible. 
We pursue this objective for two main reasons. On the one hand, by addressing structure-preservation in numerical simulations, we expect to preserve features of the solution that are usually lost by standard discretisation approaches. On the other hand, preserving the structure of the equation at the discrete level enables us to devise discretisations with similar properties for different fluid models in a unified and systematic way. 

The idea of structure preservation in numerical simulations of the Euler equations is not new. 
Probably the most successful approach in this context in terms of number of conservation laws is the sine truncation due to Zeitlin \cite{Zeitlin}, further developed in the form of a geometric integrator in \cite{McL93} by McLachlan. This uses a modified spectral truncation of the original system, yielding a Lie-Poisson structure of the equations at the discrete level. Such a truncation however is only well-defined in the very specific case of a two-dimensional periodic domain and generalisations are far from trivial. 

One of the most recent approaches is the one associated to the school of Brenier \cite{brenier89,brenier00}, see, e.g., the discretisations proposed in \cite{gallouet16,merigot16}. Such methods are based on the reformulation of the variational principle underlying the incompressible Euler equations into an optimal transport problem. This leads to the definition of the notion of generalised incompressible flow, in which the motion of the fluid particles is defined in a probabilistic sense. The methods in \cite{gallouet16,merigot16} aim at producing approximations of the flow map in this setting. In particular, they construct minimizing geodesics between discrete measure preserving maps transporting the particle mass distributions. Moreover, they establish convergence to classical solutions of the Euler equations.  

In this paper we follow a different route, which is inspired by the work of Pavlov et al.\  \cite{Pavlov,Pavlov1}. In this case, the discrete flow maps are discretised by considering their action on piecewise constant scalar functions intended as right composition.
However, such maps are not constructed explicitly and the final algorithm is fully Eulerian, i.e.\ one approximates the evolution of the velocity field at fixed positions in space, without following each fluid particle. 
One key ingredient to achieve this is the identification of vector fields with discrete Lie derivatives, i.e.\ advection operators, also acting on piecewise constant scalar functions. 
Owing to this identification, one can reduce the variational principle governing the dynamics from the space of discrete flow maps to the one of discrete velocity fields, i.e.\ from a Lagrangian (material) description to an Eulerian (spatial) one. The final algorithm can be interpreted as an advection problem for the velocity one-form by means of another discrete Lie derivative operator, defined consistently with the Discrete Exterior Calculus (DEC) formalism \cite{Desbrun03}. Desbrun et al.\ \cite{Desbrun14} extended this promising idea to generate a variational numerical scheme for rotating stratified fluid models. Both the algorithms presented in \cite{Pavlov1} and \cite{Desbrun14} can be classified as low order finite difference methods and require a dual grid construction to generate an appropriate discrete Lagrangian. 

We will be concerned with generalising the approach presented in \cite{Pavlov1} using finite element methods, with the aim of producing higher order variational schemes that require a single  computational grid. Since we are interested in discretising incompressible fluids, we need finite element vector spaces that encode the divergence-free constraint in their construction. This problem is generally solved by means of mixed finite element methods \cite{Boffi13} that effectively enable us to build such spaces as subspaces of $\bm{H}(\mr{div})$ (the Hilbert space of square-integrable vector fields with square-integrable divergence on a given domain). 

The mathematical structure of mixed problems is particularly clear when formulating them in terms of differential forms, as Arnold et al.\ \cite{Arnold06,Arnold10} did by introducing the Finite Element Exterior Calculus (FEEC) framework. In FEEC, many different mixed finite element spaces can be seen as specific instances of polynomial differential form spaces. One can analyse different discretisations at once; most importantly, stability can be addressed in a systematic way independently of the domain topology. 

From our perspective using differential forms in the discretisation will be useful to have a single definition for a discrete Lie derivative both when acting on scalar functions and differential one-forms (which can be identified with vector fields). 
Mullen et al.\ \cite{Mullen} defined a numerical approach in this direction in the context of DEC using semi-Lagrangian techniques. Heumann et al.\ \cite{Heumann,Heumann11} developed similar techniques formulated as Galerkin methods using the differential form spaces arising from FEEC. In the same work, they defined Eulerian-type discretisations including upwind stabilisation, which they applied to linear advection/advection-diffusion problems with Lipschitz-continuous advecting velocity (see also \cite{Heumann13}, for an application to the magnetic advection problem). Recently, these techniques were further developed to allow for piecewise Lipschitz-continuous advecting velocities  \cite{Heumann16}.

Our strategy will be to adopt this latter approach and apply it to the case where the vector field generating the Lie derivative is in ${\bm H}(\ddiv)$ finite element spaces. In this way, we will be able to reproduce the structure of the variational algorithm in the work of Pavlov et al. \cite{Pavlov1}, transplanting the main features to the finite element setting.
This will lead us to rediscover the centred flux discretisation described in \cite{Guzman16}, and inspired by the discretisation for the Navier-Stokes equations discussed in \cite{Cockburn05}.  Our derivation shows that such a discretisation, besides conserving kinetic energy, also inherits a discrete version of Kelvin's circulation theorem. As a matter of fact, variational formulations lead to conservation laws resulting from symmetries in Hamilton's principle, as a consequence of Noether's theorem. Unfortunately, just as in \cite{Pavlov1}, this approach does not lead to conservation of enstrophy. In turn, this yields suboptimal convergence rates as observed in \cite{Guzman16}, since a priori control of enstrophy is required for stability of smooth solutions.  

In order to overcome these issues we introduce an upwind-stabilised version of the algorithm, which is different from the one studied in \cite{Guzman16}. More specifically, we construct the method in such a way that it still possesses the same properties as the centred flux discretisation, even though it cannot be derived from first principles. We prove sub-optimal order $s$ convergence when polynomial spaces of order $s$ are used, for $s>1+n/2$, where $n$ is the domain dimension. However, numerical results suggest that, at least in two dimensions, the scheme converges optimally with rate $s+1$ for $s\geq 1$.  





This paper is organised as follows. In Section~\ref{sec:pre} we introduce the notions of one-forms and Lie derivatives, and we use these to review the variational structure underlying the incompressible Euler equations. In Section~\ref{sec:fevari} we describe our discretisation approach. In particular, we derive the central and upwind discretisations and we describe their properties.
In Section~\ref{sec:conv} we present a convergence proof for the upwind scheme, using as starting point the error estimate for the centred scheme derived in \cite{Guzman16}. We provide some numerical tests in Section~\ref{sec:num}. Finally, in Section~\ref{sec:conc} we conclude by describing further possible developments of our approach.

\section{Preliminaries}\label{sec:pre}
This section provides the basic tools necessary to formulate our numerical approach. We proceed in two steps. First, in Section \ref{sec:one-forms}, we introduce notation for one-forms and Lie derivatives. Then, in Section \ref{sec:models}, we review the continuous geometric formulation of perfect incompressible fluids and show how the equations of motion can be formulated as an advection problem for the momentum one-form. We refer to \cite{Abraham88,Marsden10} for more details on these topics.

\subsection{One-forms and Lie derivatives}
\label{sec:one-forms}
Let $V$ be a vector space. A one-form on $V$ is an element of $V^*$ the dual of $V$, i.e.\ a linear functional $a:V\rightarrow \mathbb{R}$. For any ${\bf v}\in V$, we denote by $a(\bf v)\in \mathbb{R}$ the pairing of $a$ with $\bf v$. Consider now an $m$-dimensional manifold $\Omega\subset \mathbb{R}^n$. At each point $p\in \Omega$, we have a vector space $T_p\Omega$, namely the tangent space, whose elements are the tangent vectors to $\Omega$ at $p$. The union of such spaces is the tangent bundle of $\Omega$, $T\Omega \coloneqq \bigcup_{p\in \Omega} T_p\Omega$. A vector field on $\Omega$ is a smooth assignment ${\bf v}: \Omega \rightarrow T\Omega$ such that ${\bf v}_p$ is tangent to $\Omega$ at $p$, i.e.\ ${\bf v}_p \in T_p\Omega$.  Differential one-forms generalise the concept of one-form as dual elements of vector fields on $\Omega$. Specifically, we define the cotangent space $T^*_p\Omega$ to be the dual of $T_p\Omega$, and introduce the cotangent bundle $T^*\Omega \coloneqq \bigcup_{p\in \Omega} T^*_p\Omega$. Then, a differential one-form is an assignment $a:\Omega \rightarrow T^*\Omega$ such that $p \in \Omega \rightarrow a_p\in T^*_p\Omega$. For any smooth vector field $\bf v$, the pairing $a({\bf v})$ is the real-valued function $a({\bf v}): p \in \Omega \rightarrow a_p({\bf v}_p)\in \mathbb{R}$. Then, we define the total pairing of a differential one-form $a$ with a vector field $\bf v$ to be
\begin{equation}
\langle a, {\bf v} \rangle \coloneqq \int_{\Omega} a( \bf v) \, \mr{vol}\, ,
\end{equation} 
with $\mr{vol}$ being the measure on $\Omega$ induced by the Euclidean metric. 

In this paper we restrict ourself to the special case when $\Omega$ is an $n$-dimensional domain in $\mathbb{R}^n$, with $n=2,3$, and admits global Cartesian coordinates $\{x_i\}_{i=1}^n$ oriented accordingly to an orthonormal basis $\{{\bf e}_i\}_{i=1}^n$ of $\mathbb{R}^n$. Then, a vector field $\bf v$ on $\Omega$ has the coordinate representation ${\bf v} = \sum_{i=1}^n v_i {\bf e}_i$, where $v_i$ are real-valued functions on $\Omega$. Note that ${\bf e}_i$, for $i=1,\ldots,n$, is interpreted as a constant vector field on $\Omega$. Now, let $\{\ed x_i\}_{i=1}^n$ be the dual basis to $\{{\bf e}_i\}_{i=1}^n$, i.e.\ $\ed x_i({\bf e}_j) =\delta_{ij}$. Then, a differential one-form $a$ has the coordinate representation $a = \sum_{i=1}^n a_i \ed x_i$, where $a_i$ are real-valued functions on $\Omega$, and $a({\bf v}) = \sum_{i=1}^n a_i v_i$. Note that ${\ed x}_i$, for $i=1,\ldots,n$, is interpreted as a constant differential one-form on $\Omega$. In this simple setting, we have a trivial identification of a differential one-form $a$ with a vector field $\bf a$ by the equation 
\begin{equation}
a = {\bf a}\cdot \ed {\bf x}  \, ,
\end{equation}
where the right-hand side represents the formal inner product between the vector ${\bf a} = (a_1,\ldots,a_n)$ and $\ed {\bf x} = (\ed x_1, \ldots, \ed x_n)$.
Let $\Gamma\subset \Omega$ be a smooth curve, parameterised by the functions $x_i = x_i(s)$, for $i=1,\ldots,n$ and $s\in I$, where $I$ is an open subset of $\mathbb{R}$ with $0\in I$. The quantity ${\bf s} = \sum_{i=1}^n \ed x_i/\ed s|_{s=0} \,{\bf e}_i$ is a vector field on $\Gamma$, since for each $p\in \Gamma$, ${\bf s}_p\in T_p\Gamma$. Then, the integral of a differential one-form on $\Gamma$ is defined by
\begin{equation}
\int_\Gamma a \coloneqq \int_{I} a({\bf s})\, \ed s\, ,
\end{equation} 
which is independent of the parameterisation of $\Gamma$.
A smooth vector field $\bf v$ on $\Omega$ can be identified with a one-parameter family of local diffeomorphisms  $\varphi: I \times U \rightarrow \Omega$, with $U$ being an open set of $\Omega$ and $I\subset \mathbb{R}$ an open set around $0\in \mathbb{R}$, such that, for all $t\in I$, $\varphi_t$ is a local diffeomorphism defined by
\begin{equation}
\parder{\varphi_t}{t} = {\bf v} \circ \varphi_t\,,
\end{equation}  
and $\varphi_0 = e$ the identity map on $\Omega$. The map $\varphi$ defines the flow of the vector field $\bf v$. 
The Lie derivative of a differential one-form $a$ with respect to the vector field $\bf v$ is the one-form ${\sf L}_{\bf v} a$ satisfying the equation
\begin{equation}\label{eq:lieone}
\int_{\Gamma} {\sf L}_{\bf v} a = \der{}{t}\bigg|_{t=0} \int_{\varphi_t\circ\Gamma} a \, ,  
\end{equation}
for any smooth curve $\Gamma \subset \Omega$. Moreover, if $n=3$, we have the identification
\begin{equation}\label{eq:lie3d}
 {\sf L}_{\bf v} a = ({\bf{grad}} ({\bf v}\cdot \bf a) +({\bf{curl}} \, {\bf a}) \times {\bf v}) \cdot \ed {\bf x}\,,
\end{equation}
or if $n=2$,
\begin{equation}\label{eq:lie2d}
 {\sf L}_{\bf v} a = ({\bf{grad}} ({\bf v}\cdot \bf a) +({\mr{rot}} \, {\bf a}) {\bf v}^\perp) \cdot \ed {\bf x}\,,
\end{equation}
where at each $p\in\Omega$, $\bf{v}_p^\perp$ is the vector $\bf v_p$ after a clockwise rotation of $\pi/2$, and $\mr{rot}\,\bf{a}\coloneqq\mr{div}\,\bf{a}^\perp $ (see \cite{Abraham88} for more details).
The Lie derivative can be regarded as a generalisation of the scalar advection operator. To realise this, we need to find an equivalent of Equation~\eqref{eq:lieone} for scalar functions. For a given scalar function ${b}:\Omega\rightarrow \mathbb{R}$, we  can interpret pointwise evaluation as the natural analogue of the integration of differential one-forms over curves \cite{Heumann11}. With this substitution, Equation~\eqref{eq:lieone} becomes
\begin{equation}\label{eq:lie0}
{\sf L}_{\bf v} {b} = \der{}{t}\bigg|_{t=0} {b} \circ \varphi_t\, = {\bf v}\cdot {\bf{grad}}\, b,  
\end{equation}
which is just the directional derivative of ${b}$ in the direction of $\bf v$.

We conclude this section by introducing some notation for the function spaces and relative inner products we will use in this paper. We denote by $L^2(\Omega)$ (respectively ${\bm L}^2(\Omega)$) the space of square-integrable functions (respectively vector fields) on $\Omega$. We denote by $(\cdot, \cdot)_{\Omega}$ and $\|\cdot\|_{\Omega}$ the inner product and associated norm for both $L^2(\Omega)$ and ${\bm L}^2(\Omega)$.
The definitions extend to one-forms, i.e.\ for any one-form $a={\bf a}\cdot{\ed} {\bf x}$ and $b={\bf b}\cdot{\ed} {\bf x}$,
\begin{equation}\label{eq:formpair}
(a,b)_\Omega \coloneqq \langle a, {\bf b} \rangle = ({\bf a},{\bf b})_{\Omega}\, .
\end{equation}
Moreover, we denote by $L^p(\Omega)$, with $1\leq p \leq \infty$, the standard generalisation of $L^2(\Omega)$ with norm $\|\cdot\|_{L^p(\Omega)}$. The spaces $H^k(\Omega)$ (respectively $W^{k,p}(\Omega)$) are the spaces of scalar functions with derivatives up to order $k>0$ in $L^2(\Omega)$ (respectively $L^p(\Omega)$); the relative norms are $\|\cdot\|_{H^k(\Omega)}$ and $\|\cdot\|_{W^{k,p}(\Omega)}$. Similar definitions hold for vector fields; for example, ${\bm L}^p(\Omega)$ is the space of vector fields with components in $L^p(\Omega)$. Finally, we define
\begin{equation}
{\bm H}(\mr{div},\Omega) \coloneqq \{ {\bf v}\in {\bm L}^2(\Omega)\,:\, \mr{div} \, \bf{v} \in L^2(\Omega) \}.
\end{equation}

\subsection{Geometric formulation of incompressible perfect fluids}
\label{sec:models}

In this section we review the formal variational derivation of the incompressible Euler equations. This will be important in Section~\ref{sec:dinc}, where we
develop our discretisation approach by repeating this derivation step by step in a finite dimensional setting.

The diffeomorphisms on $\Omega$ play a major role in the description of fluids as they can be used to represent the motion of particles in the absence of discontinuities such as fractures, cavitations or shocks. More specifically, denote by $\mr{Diff}(\Omega)$ the group of (boundary preserving) diffeomorphisms from $\Omega$ to itself, with group product given by map composition. The subgroup of volume preserving diffeomorphisms $\mr{Diff}_\mr{vol} (\Omega)$ is defined by
\begin{equation}
\mr{Diff}_\mr{vol} (\Omega) \coloneqq\{ g \in \mr{Diff} (\Omega)\, :\, \mr{det} (D g) = 1 \}\, ,
\end{equation}  
where $Dg$ denotes the Jacobian matrix of the map $g$.
This definition ensures that for any $g \in \mr{Diff}_\mr{vol} (\Omega)$ and any open set $U\subset \Omega$, $g(U)$ has the same measure as $U$. Then, the motion of an incompressible fluid is represented by a curve on $\mr{Diff}_\mr{vol}(\Omega)$, i.e.\ a one-parameter family of diffeomorphisms $\{g_t\}_{t\in I}$, where $I$ is an open interval of $\mathbb{R}$ containing $0\in \mathbb{R}$, so that, for all $t\in I$, $g_t \in \mr{Diff}_\mr{vol}(\Omega)$ and $g_0 = e$ the identity map on $\Omega$. This means that, for each point  $X \in \Omega$, $g_t(X)$ denotes the position occupied at time $t$ by the particle that was at $X$ when $t=0$. We define
the material velocity field,
\begin{equation}\label{eq:matvel}
\parder{g_t}{t}(X) = {\bf U} (t,X).
\end{equation}
This is a vector field on $\Omega$ that gives the velocity at time $t\in I$ of
the particle that occupied the position $X\in \Omega$ at time $t=0$,
for each $X\in \Omega$.  The spatial velocity is
the vector field
\begin{equation}
\label{eq:spatial}
{\bf u}(t,x)  \coloneqq {\bf U} (t,g_t^{-1}(x)).
\end{equation}
This gives the velocity of the particle occupying the position $x\in \Omega$ at time $t$. Therefore, in the material formulation $X\in \Omega$ is to be considered as a particle label, whereas in the spatial formulation $x\in \Omega$ is a fixed position in the physical space. 
By the fact that $g_t\in \mr{Diff}_{\mr{vol}}(\Omega)$, we obtain that any spatial velocity field is tangent to the boundary $\partial \Omega$ and has zero divergence. We also see that the spatial velocity is left unchanged when composing on the right $g_t$ with any time independent diffeomorphism $q \in \mr{Diff}(\Omega)$, i.e.\ 
\begin{equation}\label{eq:right}
\parder{}{ t}(g_t \circ q) \circ (g_t\circ q)^{-1} = \parder{}{t} g_t \circ g_t^{-1} \, .
\end{equation}
The interpretation of this is simple: composing the map $q$ with $g_t$ on the right is the same as relabelling the particles in the domain $\Omega$. On the other hand, the spatial velocity field gives the velocity at fixed positions in space and it is not dependent on the specific particle labels. 

The geometric picture of incompressible perfect fluids is based on interpreting $G\coloneqq \mr{Diff}_\mr{vol}(\Omega)$ as a configuration space, then applying the reduction techniques of dynamical systems on Lie groups \cite{Marsden10}. Strictly speaking, $\mr{Diff}_\mr{vol}(\Omega)$ is not a Lie group, and one should provide a rigorous characterisation of $\mr{Diff}_\mr{vol}(\Omega)$ to proceed in such a direction \cite{Ebin70}. Here we are not concerned with these details, and we will proceed formally.

A map $g \in G$ characterizes the configuration of the system by assigning to each particle label $X\in \Omega$ its physical position in the domain $x = g(X)$. 
Let ${\bf U} \in T_g G$. Then there exists a one-parameter family $\{g_t\}_{t\in I}$, defined as above, such that, for a fixed $\bar{t}\in I$, $g_{\bar t} = g$ and $\parder{}{t}|_{t=\bar{t}}\, g_t = {\bf U}$. By comparison with equation~\eqref{eq:matvel}, $\bf U$ can be interpreted as a material velocity field on $\Omega$.
In the classical Lagrangian setting the evolution of the system is described by a curve on $TG$. In our case, this would be equivalent to the material formulation where we follow the motion as a time-dependent diffeomorphism $g_t$ together with the associated material velocity field as defined in Equation~\eqref{eq:matvel}.

The tangent space at the identity $\mathfrak{g}\coloneqq T_eG$ is the Lie algebra of $G$. By Equation~\eqref{eq:spatial}, we can regard elements of $\mf{g}$ as spatial velocity fields on $\Omega$. Therefore, an element of $\mathfrak{g}= T_e G$ can be identified with a vector field on $\Omega$ tangent to its boundary $\partial \Omega$, and with zero divergence.  For all ${\bf u},{\bf v}\in \mf{g}$, we denote by $({\bf u},{\bf v}) \mapsto \mr{ad}_{\bf u} {\bf v}$ their Lie algebra product. Let $q_s$ be a curve on $G$ such that $q_0 = 0$ and $\dot{q}_0 = \bf v$, the map $\mr{ad}:\mf{g} \times \mf{g} \rightarrow \mf{g}$ is defined by
\begin{equation}\label{eq:adlie}
\mr{ad}_{\bf u} {\bf v} \coloneqq  \der{}{t}\bigg|_{t=0}\der{}{s}\bigg|_{s=0} {g_t}\circ{q_s}\circ g_t^{-1} =  \der{}{t}\bigg|_{t=0} (D {g_t}\,{\bf v}) \circ  g_t^{-1} =  - [{\bf u}, {\bf v}] \, ,
\end{equation}
where $[{\bf u}, {\bf v}]$ is the Jacobi-Lie bracket, defined in coordinates by
\begin{equation}\label{eq:ljbrac}
[{\bf u}, {\bf v}] \coloneqq \sum_{i,j=1}^n\left( u_j \parder{v_i}{x_j}- v_j \parder{u_i}{x_j}\right) {\bf e}_i\, .
\end{equation}

The Lagrangian of the system is a function $L:TG\rightarrow \mathbb{R}$, and is given by the kinetic energy of the spatial velocity field, i.e.\
\begin{equation}\label{eq:lag}
L(g,\dot{g}) = \frac{1}{2} \int_{\Omega} \|\dot{g}\circ g^{-1}\|^2 \, \mr{vol} \, .
\end{equation}
Such a Lagrangian is right invariant because of Equation~\eqref{eq:right}.
Therefore, the dynamics can be expressed in terms of the reduced Lagrangian $l({\bf u})\coloneqq L(e, {\bf u})$ only. More precisely,
Hamilton's principle for the Lagrangian in Equation~\eqref{eq:lag} is equivalent to
\begin{equation}\label{eq:varpri}
\delta \int_{t_0}^{t_1} l( {\bf u} ) \,\ed t = 0\, ,
\end{equation}
with variations $\delta {\bf u} = \dot{{\bf v}} - \mr{ad}_{{\bf u}} {\bf v}$, where ${\bf v}_t \in \mathfrak{g}$ is an arbitrary curve and ${\bf v}_{t_0}= {\bf v}_{t_1}=0$. The variational principle in Equation~\eqref{eq:varpri} produces the Euler-Poincar\'e equations,
\begin{equation}\label{eq:EPfluid}
\left\langle\der{}{t} \frac{\delta l}{\delta {\bf u}} , {\bf v} \right\rangle + \left\langle \frac{\delta l}{\delta {\bf u}}, \mr{ad}_{\bf u} {\bf v}  \right\rangle = 0\, ,
\end{equation}
for all ${\bf v}\in \mathfrak{g}$, where we define
\begin{equation}
\left\langle \frac{\delta l}{\delta {\bf u}} , {\bf v} \right\rangle\coloneqq \der{}{\epsilon}\bigg|_{\epsilon = 0} l({\bf u}+ \epsilon {\bf v})\, .
\end{equation}
Now $\delta l/\delta \bf u$ can be identified with a one-form $m$ on $\Omega$, which we will refer to as momentum, imposing that
\begin{equation}\label{eq:dlfluid}
\left\langle \frac{\delta l}{\delta {\bf u}} , {\bf v} \right\rangle
= \int_{\Omega} {m} ({\bf v})\, \mr{vol}\, ,
\end{equation}
for all ${\bf v}\in \mathfrak{g}$. In our case, we readily see that ${m} = u \coloneqq {\bf u}\cdot{\ed {\bf x}}$, however we will distinguish between momentum one-forms ${m}$ and velocity one-forms $u$ since this will be useful when describing model problems characterised by more complex Lagrangians. Inserting Equations~\eqref{eq:adlie} and \eqref{eq:dlfluid} into Equation~\eqref{eq:EPfluid} we obtain
\begin{equation}\label{eq:euler1}
\int_{\Omega} \dot{{m}} ({\bf v}) \,\mr{vol}-\int_{\Omega} {m} ([{\bf u},{\bf v}]) \,\mr{vol}=0\,.
\end{equation}
If $n=3$, for divergence free vector fields $\bf u$ and $\bf v$ we have the useful identity $[{\bf u}, {\bf v}] = - {\bf{curl}}({\bf u} \times {\bf v})$. Hence, integrating by parts, the second integral can be rewritten as follows,
\begin{equation}\label{eq:intpartlie}
\begin{aligned}
\int_{\Omega} {m}([{\bf u},{\bf v}]) \,\mr{vol} & = -
\int_{\Omega} {\bf m} \cdot {\bf {curl}} ({\bf u} \times {\bf v})\, \mr{vol}\\
& = \int_{\Omega} ( {\bf u} \times {\bf{ curl}} \,{\bf m}) \cdot {\bf v} \, \mr{vol}  \\
& = - \int_{\Omega}  {\sf L}_{\bf u} {m} ({\bf v})\,\mr{vol}\, . 
\end{aligned}
\end{equation}
The same result holds for $n=2$.
Note that the ${\bf{grad}}$ term in Equation~\eqref{eq:lie3d} or \eqref{eq:lie2d} vanishes upon pairing with $\bf v$ because $\mr{div}\,{\bf v} =0$. Reinserting Equation~\eqref{eq:intpartlie} into Equation~\eqref{eq:euler1} yields the incompressible Euler equations in variational form
\begin{equation}\label{eq:euler}
\langle \dot{{m}}, {\bf v} \rangle + \langle {\sf L}_{\bf u} {m}, {\bf v} \rangle = 0\, ,
\end{equation}
for all ${\bf v} \in \mathfrak{g}$, where $m=u\coloneqq \bf u \cdot {\ed {\bf x}}$.

\begin{remark}
\label{rem:epadv}
Equation~\eqref{eq:intpartlie} establishes a link between the bracket on $\mathfrak{g}$ coming from Lagrangian reduction and the advection of the momentum one-form ${m}$ by the Lie derivative operator. In the following we will develop a discretisation approach that maintains this property and in this way provides a discrete version of both the bracket and the Lie derivative.
\end{remark}

\section{Finite element variational discretisation}\label{sec:fevari}
In this section we construct a variational algorithm for the incompressible Euler equations using discrete Lie derivatives intended as finite element operators as the main ingredient. As starting point for this, we describe the Galerkin discretisation of Lie derivatives proposed by \cite{Heumann11} (see also \cite{Heumann13}), in the case where the advecting velocity is a smooth vector field. Then, in Section~\ref{sec:dlie}, we consider a generalisation for advecting velocities in finite element spaces (meaning that the velocities have reduced smoothness). Such a generalisation coincides with the one proposed in \cite{Heumann16} for the more general case of piecewise Lipschitz-continuous advecting velocities. Next, in Section~\ref{sec:dinc}, we construct a variational algorithm using these operators as discrete Lie algebra variables, adapting the approach of \cite{Pavlov1} to the finite element setting. Moreover, we show that such an algorithm coincides with the centred flux discretisation described in \cite{Guzman16}. In Section~\ref{sec:upwind}, we include upwind stabilisation to the method derived in the previous section.
Finally, in Section~\ref{sec:dkel}, we show that both the centred flux and upwind schemes conserve energy and possess a discrete version of Kelvin's circulation theorem.
Throughout this section we assume $n=3$, although all the results of this section extend without major modifications to the case $n=2$.

\subsection{Discrete Lie derivatives}\label{sec:dlie}
Let ${\mc T}_h$ be a triangulation on $\Omega$, i.e.\ a decomposition of $\Omega$ into simplices $K\subset \Omega$, with $h$ being the maximum simplex diameter. We define $V_h\subset L^2(\Omega)$ to be a chosen scalar finite element space on ${\mc T}_h$, and ${\bf W}_h \subset{\bm H}(\mr{div},\Omega)$ to be a chosen ${\bm H}(\mr{div},\Omega)$-conforming finite element space of vector fields on $\Omega$. We define $W_h\coloneqq {\bf W}_h\cdot{\ed {\bf x}}$, i.e.\ $W_h$ is the set of one-forms $a = {\bf a} \cdot \ed {\bf x}$ with $ {\bf a} \in {\bf W}_h$.

More specifically, for $r\geq 0$, we denote by $V^r_h\subset L^2(\Omega)$ a scalar finite element space such that, for all elements $K\in {\mc T}_h$, $V^r_h|_K = \mc P_{r}(K)$  where $\mc P_r(K)$ is the set of polynomial functions on $K$ of degree up to $r$. In other words, $V_h^r$ is the standard scalar discontinuous Galerkin space of order $r$. Analogously, we denote by  ${\bf W}^r_h \subset {\bm H}(\mr{div},\Omega)$ either the Raviart-Thomas space of order $r$, defined  for $r\geq 0$ by
\begin{equation}\label{eq:RT}
\bf{RT}_r(\mc T_h) \coloneqq \{ {\bf u}\in {\bm H}(\mr{div},\Omega) ~:~ {\bf u}|_K \in (\mc{P}_r(K))^n +{\bf x} \mc{P}_r(K) \}\,,
\end{equation}
or the or Brezzi-Douglas-Marini space of order $r$, and defined for $r\geq 1$ by
\begin{equation}
\bf{BDM}_r(\mc T_h)\coloneqq \{ {\bf u}\in {\bm H}(\mr{div},\Omega) ~:~ {\bf u}|_K \in (\mc{P}_r(K))^n \} \,.
\end{equation}
Later, we will use the divergence-free subspace of ${\bf W}^r_h$, which is the same for  $\bf{RT}_r(\mc T_h)$ and $\bf{BDM}_r(\mc T_h)$ (see, e.g., Corollary 2.3.1 in \cite{Boffi13}).

Let $\mc F$ be the sets of $(n-1)$-dimensional simplices in ${\mc T}_h$, and let ${\mc F}^\circ\subset  {\mc F}$ be the set of facets $f\in {\mc F}$ such that $f \cap \partial \Omega = \emptyset$. We fix an orientation for all $f\in {\mc F}$ by specifying the unit normal vector ${\bf n}_f$ in such a way that on the boundary ${\bf n}_f$ is equal to the outward pointing normal ${\bf n}_{\partial \Omega}$. An orientation on the facet $f\in {\mc F}$ defines a positive and negative side on $f$, so that any $a\in V_h$ has two possible restrictions on $f$, denoted $a^+$ and $a^-$, with $a^+\coloneqq a|_{K^+}$ and $K^+\in \mc T_h$ being the element with outward normal $\bf{n}_f$. Therefore, we can define for each $a \in V_h$ and on each $f\in \mc F^\circ$ the jump operator  $[a] \coloneqq a^+ - a^-$ and the average operator $\{a\} \coloneqq (a^+ + a^-)/2$. Such definitions extend without modifications to vector fields ${\bf a} \in {\bf W}_h$ or to their one-form representation $a = {\bf a}\cdot {\ed {\bf x}}\in W_h$. 

\begin{lemma}\label{lem:hdivc}
Let ${\bf W}_h\subset {\bm H}(\mr{div},\Omega)$, then for all ${\bf a}\in {\bf W}_h$ and for all $f\in \mc F$, ${\bf a}\cdot {\bf n}_f$ is single-valued on $f$, and in particular $[{\bf a}]\cdot {\bf n}_f = 0$.
\end{lemma}
\begin{proof}
See Lemma 5.1 in \cite{Arnold06}.
\end{proof}
For all elements $K\in \mc T_h$ and $f\in \mc F$, we denote by $(\cdot,\cdot)_K$ and $(\cdot,\cdot)_f$ the standard $L^2$ inner products on $K$ and $f$. We use the same definition for both scalar and vector quantities, and we extend it to one-forms as in Equation~\eqref{eq:formpair}. 
Furthermore, given a smooth vector field $\bm \beta$, for any $f\in \mc F$ and for any smooth scalar functions $a$ and $b$ on $f$ we define 
\begin{equation}(a,b)_{f,{\bm \beta}} \coloneqq ( {\bm \beta}\cdot{\bf n}_f\, a, b)_f\,.
\end{equation}
Similarly, for any element $K\in\mc T_h$ and for any smooth scalar functions $a$ and $b$ on $\partial K$ define 
\begin{equation}(a,b)_{\partial K,{\bm \beta}} \coloneqq ( {\bm \beta}\cdot{\bf n}_{\partial K}\, a, b)_{\partial K}\,,
\end{equation}
where ${\bf n}_{\partial K}$ is the unit normal to the boundary of the element $K$ pointing outwards. 
The same definitions hold if $a$ and $b$ are one-forms or vector fields with smooth components on $f$ or $\partial K$.

A discrete Lie derivative in this paper means any finite element operator from $V_h$ to itself, or from $W_h$ to itself, which is also a consistent discretisation of the Lie derivative. 
In \cite{Heumann11} the authors proposed two types of discrete Lie derivatives\footnote{The discrete Lie derivatives proposed in \cite{Heumann11} are more general than considered in this paper, since they account for a larger class of finite element spaces, either conforming or not. In the language of differential forms, we use spaces that are conforming with respect to the $L^2$-adjoint of the exterior derivative; this case is considered explicitly in Section 4.1.1 of \cite{Heumann}.} based on either an Eulerian or a semi-Lagrangian approach, which were used in the context of linear advection and advection-diffusion problems. 
We will take as starting point their Eulerian discretisation (see Section 4 in \cite{Heumann11}), where the advecting velocity is a fixed smooth vector field $\bm \beta$.
In the following, we will consider extensions of this approach to the case where $\bm \beta$ is not fully continuous; specifically $\bm \beta \in {\bf W}_h\subset {\bm H}(\mr{div},\Omega)$. The resulting operators coincide with the ones proposed in \cite{Heumann16} for piecewise Lipschitz-continuous advecting velocities. Such a discussion will be useful to discretise the nonlinear advection term in the Euler equations. The Eulerian-type discrete Lie derivative proposed by \cite{Heumann11} is given by the following definition.

\begin{definition}\label{def:dLie}
Let ${\bm \beta}$ be a smooth vector field on $\Omega$ tangent to the boundary $\partial \Omega$. The finite element operator ${\sf L}_{\bm \beta}^h: V_h \rightarrow V_h$ (respectively ${\sf L}_{\bm \beta}^h: { W}_h \rightarrow {W}_h$) is defined by 
\begin{equation}\label{eq:clas}
({\sf L}_{\bm \beta}^h a, b)_{\Omega} = \sum_K ({\sf L}_{\bm \beta} a, b)_{K}  + \sum_{f\in \mc F^\circ} \Big(-([a ], \{b\})_{f,{\bm \beta}} +  ([a], [b])_{f,c_f {\bm \beta}}\Big)\,,
\end{equation}
where $c_f:f \rightarrow \mathbb{R}$ is a scalar function depending on $\bm \beta$, for all $a,b\in V_h$ (respectively $a,b \in W_h$).
\end{definition}
The consistency of such a discretisation is proved by observing that if the advected quantity $a$ is smooth, ${\sf L}^h_{\bm \beta} a$ coincides with the Galerkin projection of ${\sf L}_{\bm \beta}a$ onto $V_h$ or $W_h$. Note that the Lie derivative discretisation on $V_h$ as defined in Equation~\eqref{eq:clas} coincides with the classical discontinuous Galerkin discretisation of the advection operator described in \cite{Brezzi04}. This can be verified using integration by parts in Equation~\eqref{eq:clas}, yielding the following equivalent definition for the operator ${\sf L}_{\bm \beta}^h: V_h \rightarrow V_h$,  
\begin{equation}\label{eq:clas1}
({\sf L}_{\bm \beta}^h a, b)_{\Omega} = -\sum_K ( a, \ddiv({ \bm \beta}\,b))_{K}  + \sum_{f\in \mc F^\circ} \Big((\{a \}, [b])_{f,{\bm \beta}} +  ([a], [b])_{f,c_f {\bm \beta}}\Big)\,,
\end{equation}
for all $a,b\in V_h$. Similarly, the operator ${\sf L}_{\bm \beta}^h: W_h \rightarrow W_h$ can be equivalently defined by
\begin{equation}\label{eq:clas11}
	({\sf L}_{\bm \beta}^h a, b)_{\Omega}= \sum_K ( {\bf a}, {\bf{curl}} ({\bm \beta}\times {\bf b}) -{\bm \beta} \,\mr{div}\,{\bf b})_{K} +\sum_{f\in \mc F^\circ} \Big((\{ {\bf a} \}, [\bf{b}])_{f,{\bm \beta}} +  ([\bf{a}], [\bf{b}])_{f,c_f {\bm \beta}}\Big)\,,
\end{equation}
for all $a,b\in W_h$ with $a = {\bf a}\cdot {\ed }{\bf x}$ and $b = {\bf b} \cdot {\ed{\bf x} }$. Details on the calculations leading to Equations~\eqref{eq:clas1} and \eqref{eq:clas11} can be found in \cite{Heumann}.

 
\begin{remark}\label{rem:upwind0} 
There are two choices for the function $c_f$ that are particularly important:
\begin{itemize}
\item $c_f = 0$ (centred discretisation). In this case, Equation~\eqref{eq:clas1} reduces to 
\begin{equation}\label{eq:clasc}
({\sf L}_{\bm \beta}^h a, b)_{\Omega}= -\sum_K (a, \mr{div} ({\bm \beta}\, b))_{K}  + \sum_{f\in \mc F^\circ}  (\{a\}, [b])_{f,{\bm \beta}}\,,
\end{equation}
for all $a,b\in V_h$. We refer to such discretisation as centred, since the facet integrals contain the average of the advected quantity $a$;
\item  $c_f = {\bm \beta} \cdot{\bf n}_f /(2|{\bm \beta} \cdot{\bf n}_f|)$ (upwind discretisation). In this case, Equation~\eqref{eq:clas1} reduces to
\begin{equation}\label{eq:clasup}
({\sf L}_{\bm \beta}^h a, b)_{\Omega}= -\sum_K (a, \mr{div} ({\bm \beta}\, b))_{K}  + \sum_{f\in \mc F^\circ}  (a^{up}, [b])_{f,{\bm \beta}}\,,
\end{equation}
for all $a,b\in V_h$, where $a^{up} = a^+$ if ${\bm \beta}\cdot {\bf n}_f\geq0$ and $a^{up} = a^-$ if ${\bm \beta}\cdot {\bf n}_f<0$, i.e.\ in the facet integrals $a$ is always evaluated from the upwind side. 
\end{itemize}
A similar interpretation holds when considering the action of ${\sf L}_{\bm \beta}^h$ on $W_h$.
\end{remark}


\begin{lemma}\label{lem:alt}
Let $n=3$ and let ${\bm \beta}$ be a smooth vector field on $\Omega$ tangent to the boundary $\partial \Omega$. Then, the discrete Lie derivative ${\sf L}^h_{\bm \beta}:W_h\rightarrow W_h$ in Equation~\eqref{eq:clas} can be equivalently defined by 
\begin{equation}\label{eq:alt}
\begin{aligned}
	({\sf L}_{\bm \beta}^h a, b)_{\Omega}= &\sum_K ( {\bf a}, {\bf{curl}} ({\bm \beta}\times {\bf b}) -{\bm \beta} \,\mr{div}\,{\bf b})_{K}  \\&+ \sum_{f\in \mc F^\circ} \Big(( {\bf n}_f\times\{{\bf a}\}, {\bm \beta} \times [{\bf b}])_{f}+  (c_f\, {\bf n}_f\times[{\bf a}], {\bm \beta} \times [{\bf b}])_{f}\Big)\,,
\end{aligned}
\end{equation}
for all $a,b\in W_h$ with $a = {\bf a}\cdot {\ed }{\bf x}$ and $b = {\bf b} \cdot {\ed{\bf x} }$.
\end{lemma}
\begin{proof}
The proof relies on the fact that if ${\bf a} \in {\bf W}_h\subset{\bm H}(\mr{div},\Omega)$ then, by Lemma~\ref{lem:hdivc}, ${\bf a}\cdot {\bf n}_f$ is single-valued across elements of $\mc T_h$, so $[{\bf a}]\cdot {\bf n}_f = 0$ on all facets $f\in \mc F^\circ$.
Then, the result follows by applying to the facet integrals in Equation~\eqref{eq:clas11} the vector identity 
\begin{equation}
(\bf{A} \times {\bf B}) \cdot (\bf{C} \times {\bf D})= (\bf{A}\cdot\bf{C})(\bf{B}\cdot\bf{D})-(\bf{B}\cdot\bf{C})(\bf{A}\cdot\bf{D}) \,,
\end{equation}
for all $\bf{A},\bf{B},\bf{C},\bf{D}\in \mathbb{R}^3$. See \cite{Heumann} for details.
\end{proof}

As we are concerned with discretising the Euler equations, we want to extend Definition~\ref{def:dLie} to the case where the advecting velocity also belongs to a finite element space. We will limit ourself to the case ${\bm \beta} \in {\bf W}_h\subset {\bm H}(\mr{div},\Omega)$. This will prove to be enough for our scope since if ${\bm \beta} \in {\bm H}(\mr{div},\Omega)$, we can ensure that $\mr{div}\, {\bm \beta} = 0$. 

\begin{remark}\label{rem:hdiv}
Equation~\eqref{eq:clas} is still well-defined if ${\bm \beta} \in {\bf W}_h$, since in this case, by Lemma~\ref{lem:hdivc}, ${\bm \beta} \cdot {\bf n}_f$ is single-valued on each facet $f \in {\mc F}^\circ$.
\end{remark}
 
In view of Remark~\ref{rem:hdiv}, Equation~\eqref{eq:clas}  can be used without modification to extend the definition of the discrete Lie derivative to the case ${\bm \beta} \in {\bf W}_h$. However, this is not the only possible extension: there exist many other discretisation approaches that reduce to the one of Definition~\ref{def:dLie} for continuous advecting velocity. In particular, Lemma~\ref{lem:alt} suggests the following alternative definition, which is also the one adopted in \cite{Heumann16}.


\begin{definition}\label{def:dLiehdiv}
Let $n=3$ and let ${\bm \beta}\in{\bf W}_h$ such that ${\bm \beta} \cdot{\bf n}_{\partial \Omega} =0$ on $\partial \Omega$. The finite element operator ${\sf X}_{\bm \beta}^h:V_h \rightarrow V_h$ is defined by ${\sf X}_{\bm \beta}^h \,a \coloneqq {\sf L}_{\bm \beta}^h \, a$ for all $a \in V_h$; whereas the operator ${\sf X}_{\bm \beta}^h:W_h \rightarrow W_h$ is defined by
\begin{equation}\label{eq:clas2}
\begin{aligned}
({\sf X}_{\bm \beta}^h a, b)_{\Omega}= &\sum_K ( {\bf a}, {\bf{curl}} ({\bm \beta}\times {\bf b}) -{\bm \beta} \,\mr{div}\,{\bf b})_{K}  \\&+ \sum_{f\in \mc F^\circ} \Big(({\bf n}_f\times\{{\bf a}\}, [{\bm \beta} \times {\bf b}])_{f}+  (c_f\, {\bf n}_f\times[{\bf a}], [{\bm \beta} \times {\bf b}])_{f}\Big)\,,
\end{aligned}
\end{equation}
where $c_f:f \rightarrow \mathbb{R}$ is a scalar function depending on $\bm \beta$, for all $a,b\in W_h$ with $a ={\bf a} \cdot \ed {\bf x}$ and $b ={\bf b} \cdot \ed {\bf x}$. In other words, ${\sf X}_{\bm \beta}^h$ coincides with ${\sf L}^h_{\bm \beta}$ when applied to scalar functions but not when applied to one-forms.
\end{definition}

\begin{proposition}
For a continuous vector field $\bm \beta$, we have the equivalence ${\sf X}_{\bm \beta}^h={\sf L}_{\bm \beta}^h$. Therefore, ${\sf X}_{\bm \beta}^h$ is a discrete Lie derivative, i.e. it is a consistent discretisation of the Lie derivative operator. 
\end{proposition}
\begin{proof}
This is immediate by comparing Equation~\eqref{eq:alt} with  Equation~\eqref{eq:clas2} and observing that if $\bm \beta$ is continuous, for any ${\bf b}\in {\bf W}_h$, we have ${\bm \beta} \times [{\bf b}] = [{\bm \beta} \times {\bf b}] $ on all facets $f\in {\mc F}^\circ$.
\end{proof}

\subsection{Discrete incompressible Euler equations}\label{sec:dinc}
We now use the discrete Lie derivatives defined in the previous section to design a variational discretisation approach for the incompressible Euler equations. The derivation presented here closely follows \cite{Pavlov1}, but is adapted to the context of finite element discretisations, and does not require a dual grid. First, we will define a discrete group that approximates the configuration space $\mr{Diff}_{\mr{vol}}(\Omega)$. Next, we will derive the relative Euler-Poincar\'e equations for an appropriately chosen Lagrangian. Finally, we will connect the resulting algorithm to the Lie derivative discretisations in Definition~\ref{def:dLiehdiv}. Such an approach will lead us to rediscover the centred flux discretisation proposed in \cite{Guzman16}. 

The main issue related to deriving a variational integrator for the incompressible perfect fluids is that the configuration space $G = \mr{Diff}_{\mr{vol}}(\Omega)$ is infinite dimensional, so one needs to find an appropriate finite dimensional approximation that converges, in the limit, to the original system. In \cite{Pavlov1}, this issue is solved by applying Koopman's lemma, i.e.\ identifying group elements with their action, intended as right composition, on $L^2$ functions on $\Omega$. This is to say that $G$ can be equivalently represented as a subgroup $ G(V) \subset GL ( V )$ of the group of invertible linear maps from $V\coloneqq L^2(\Omega)$ to itself, by means of a group homomorphism $\rho:G \rightarrow G(V)$ defined by
\begin{equation}
\rho(g) \cdot a = \rho_g \cdot a := a \circ g^{-1} \, ,
\end{equation}
for any $g\in G$ and $a\in C^\infty(\Omega) \subset V$.  As the Lie algebra $\mf g$ of $G$ is the set of divergence-free vector fields $\bf v$ tangent to the boundary, the Lie algebra $\mf g(V)$ of $G(V)$ is the set of Lie derivatives ${\sf L}_{\bf v}: V \rightarrow V$, interpreted as unbounded operators, with respect to such vector fields. As a matter of fact, we have
\begin{equation}\label{eq:lie}
\der{}{s}\bigg|_{s=0} \rho_{g_s} \cdot a = \der{}{s}\bigg|_{s=0} (a \circ g^{-1}_s) = -{\sf L}_{\bf v} a \,,
\end{equation}
for any $a\in C^\infty(\Omega) \subset V$, and where $g_s$ is a curve on $G$ such that $g_0=e$ and $\dot{g_s}|_{s=0} = {\bf v}$.

In order to get a discrete version of this picture, we start by restricting $V$ to be the finite element space $V_h$.
Therefore, we consider the finite dimensional Lie group $GL(V_h)$ and we seek an appropriate subgroup $G_h(V_h)$ that in the limit approximates $G(V)$. Clearly, this cannot be accomplished by restricting $G(V)$ to $V_h$, as we have done for the general linear group $GL(V)$, since it would imply losing the group structure. However, one can still construct a subgroup $G_h(V_h)\subset GL(V_h)$ that approximates $G(V)$ and hence $G$. 

For all $g^h\in GL(V_h)$ and $a\in V_h$ we denote by $a \mapsto g^h a$ the action of $g^h$ on $a$. Then, $G_h(V_h)$ can be defined as follows.

\begin{definition} The Lie group $G_h(V_h)\subset GL(V_h)$ is the subgroup of the group of linear invertible operators from $V_h$ to itself, defined by the following properties: for all $g^h\in G_h(V_h)$,
\begin{align}\label{eq:prop1}
&(g^h a, g^h b)_{\Omega} = (a,b)_{\Omega} &&\forall\, a,b \in V_h\,, \\
 & g^h c = c  &&\forall\, c\in \mathbb{R}\, ,\label{eq:prop2}
\end{align}
where $c$ is considered as a constant function on $\Omega$.
\end{definition}

It is trivial to check that both Equation~\eqref{eq:prop1} and \eqref{eq:prop2} are verified in the continuous case. In particular for any $g\in \mr{Diff}(\Omega)$, we have $c \circ g^{-1} = c$. Moreover, as a consequence of volume preservation, for all $a,b \in V$ and $g\in \mr{Diff}_{\mr{vol}}(\Omega)$ ,
\begin{equation}
\int_{\Omega} \, (a\circ g^{-1})(b\circ g^{-1})\, \mr{vol} = 
\int_{\Omega} \, (a b)\circ g^{-1}\, \mr{vol} =
\int_{\Omega} \, \mr{det}(Dg^{-1}) a b \, \mr{vol} =
\int_{\Omega} \, a b \, \mr{vol} \,,
\end{equation}
since $g^{-1}(\Omega) = \Omega$.

The Lie algebra of $GL(V_h)$ is the set $\mf{gl}(V_h)$ of all linear operators from $V_h$ to itself. For all ${\sf A}^h\in \mf{gl}(V_h)$ and $a\in V_h$ we denote by $a \mapsto {\sf A}^h a$ the action of ${\sf A}^h$ on $a$. Then, the Lie algebra of $G_h(V_h)$ can be defined as follows. 

\begin{definition} The Lie algebra $\mf{g}_h(V_h)\subset \mf{gl}(V_h)$ associated to the Lie group $G_h(V_h)\subset GL(V_h)$ is the subalgebra of the linear operators from $V_h$ to itself defined by the following properties: for all ${\sf A}^h \in \mf{g}_h(V_h)$,
\begin{align}\label{eq:prop1a}
&({\sf A}^h a, b)_{\Omega} = -( a, {\sf A}^h b)_{\Omega}  &&\forall\, a,b \in V_h\,, \\
& {\sf A}^h c = 0  &&\forall\, c\in \mathbb{R}\, .\label{eq:prop2a}
\end{align}
\end{definition}

Note that Equation~\eqref{eq:prop1a} and \eqref{eq:prop2a} can be derived directly from Equation~\eqref{eq:prop1} and \eqref{eq:prop2}. In particular, Equation~\eqref{eq:prop1a} can be again related to volume preservation. Moreover, we can regard $\mf{g}_h(V_h)$ as an approximation $\mf{g}(V)$ so its element are to be thought as discrete Lie derivatives in view of Equation~\eqref{eq:lie}.

The collection of all the spaces $G_h(V_h)$, for $h$ arbitrarily
small, is not $G(V)$, even if the union of the spaces $V_h$ is dense
in $V$. This is because there exist elements of $GL(V)$ that satisfy
Equations~\eqref{eq:prop1} and \eqref{eq:prop2} but do not belong to $G(V)$ (for example, they might be representative of maps that are not even continuous). Therefore, before
defining the discrete system we must restrict ourselves to a space
smaller than $G_h(V_h)$.

The same considerations hold at the Lie algebra level, meaning that one needs to constrain the dynamics in a subset $S_h(V_h) \subset \mf g_h(V_h)$. Our definition for such a space is based on the discrete Lie derivatives of Definition~\ref{def:dLiehdiv}. In particular, we fix a finite element space ${\bf W}_h \subset {\bm H}(\mr{div},\Omega)$, and define
\begin{equation}\label{eq:W}
\Wh \coloneqq \{ {\bf u} \in {\bf W}_h \,:\, \mr{div}\,{\bf u} =0 \, ,\, {\bf u}\cdot{\bf n}_{\partial \Omega} =0\}.
\end{equation}
We denote by $\wh\coloneqq \Wh\cdot {\ed {\bf x}}$, i.e.\  $\wh$ is the set of one-forms $u = {\bf u} \cdot \ed {\bf x}$ with ${\bf u} \in \Wh$. Analogous definitions hold for $\Wh^r$ and $\wh^r$. Note that $\Wh^r$ is the same space if constructed using ${\bf W}_h^r = \bf{RT}_r(\mc T_h)$ or ${\bf W}_h^r = \bf{BDM}_r(\mc T_h)$ (see, e.g., Corollary 2.3.1 in \cite{Boffi13}).
The set of advection operators associated to $\Wh$ is given by
\begin{equation}
S_h(V_h) \coloneqq \{ {\sf X}_{\bf u}^h:V_h \rightarrow V_h \,~\text{with}~ c_f=0 \, : \, {\bf u} \in \Wh\}.
\end{equation}
More specifically, we define
\begin{equation}
S_h^s(V_h^r) \coloneqq \{ {\sf X}_{\bf u}^h:V_h^r \rightarrow V_h^r \,~\text{with}~ c_f=0 \, : \, {\bf u} \in \Wh^s\}.
\end{equation}
Note that setting $c_f=0$ we immediately ensure that $S_h(V_h)$ is a linear space; this can be directly verified using Definition~\ref{def:dLiehdiv}. 

\begin{lemma}\label{lem:subset}
$S_h(V_h) \subset \mathfrak{g}_h(V_h)$.
\end{lemma}
\begin{proof}
We need to verify that the elements of $S_h(V_h)$ satisfy Equation~\eqref{eq:prop1a} and \eqref{eq:prop2a}.
By Definition~\ref{def:dLiehdiv} and Equation~\eqref{eq:clas}, for all $a,b\in V_h$,
\begin{equation}\label{eq:defX}
({\sf X}_{\bf u}^h a, b)_{\Omega}= \sum_K ({\sf L}_{\bf u} a, b)_{K}  - \sum_{f\in \mc F^\circ} ([a ], \{b\})_{f,{\bf u}}\,.
\end{equation}
The Lie derivative is just the directional derivative when applied to scalar functions, as shown in Equation~\eqref{eq:lie0}. Moreover, since $\mr{div}\, {\bf u}=0$, integration by parts yields
\begin{equation}\label{eq:intpart}
\begin{aligned}
\sum_K ({\sf L}_{\bf u} a, b)_{K}& = -\sum_K ( a, {\sf L}_{\bf u}b)_{K} + \sum_K (a,b)_{\partial K,{\bf u}}\\
&  = -\sum_K ( a, {\sf L}_{\bf u} b)_{K} + \sum_{f \in{\mc F^\circ}} \Big((a^+,b^+)_{f,{\bf u}} -(a^-,b^-)_{f,{\bf u}}\Big) \, .
\end{aligned}
\end{equation}
Note that the decomposition of the integrals on $\partial K$ is justified by the fact that each facet is shared by the boundary of two adjacent elements, therefore it appears twice in the sum but with two opposite normals (since ${\bf n}_{\partial K}$ is always oriented outwards). Moreover, we have the identity
\begin{equation}\label{eq:jid}
(a^+,  b^+)_{f,{\bf u}} -  (a^-,  b^-)_{f,{\bf u}} = (\{a\},  [b])_{f,{\bf u}} +  ([a],  \{b\})_{f,{\bf u}}\, .
\end{equation}
Combining Equation~\eqref{eq:defX}, \eqref{eq:intpart} and \eqref{eq:jid} gives
\begin{equation}
({\sf X}_{\bf u}^h a, b)_{\Omega} = -( a, {\sf X}_{\bf u}^h b)_{\Omega}\, ,
\end{equation}
which is the same as in Equation~\eqref{eq:prop1a}. Equation~\eqref{eq:prop2a} is immediately verified by observing that if we set $a=c$ in Equation~\eqref{eq:defX}, with $c \in \mathbb{R}$ being a constant function on $\Omega$, we obtain
\begin{equation}
({\sf X}_{\bf u}^h c, b)_{\Omega} = 0\, ,
\end{equation}
for all $b\in V_h$, and therefore ${\sf X}_{\bf u}^h c =0$.
\end{proof}
\begin{lemma}\label{lem:iso} Let $\Omega$ be a domain in $\mathbb{R}^n$ and let $\mc T_h$ be a triangulation on $\Omega$. For any $r\geq s$ and $r\geq1$, the spaces $S_h^s(V^r_h)$ and $\Wh^s$ are isomorphic. 
\end{lemma}
\begin{proof}
We assumed that $\{ {\bf e}_i\}_{i=1}^n$ is a global orthonormal reference frame on $\mathbb{R}^n$ with coordinates $\{x_i\}_{i=1}^n$. Then, since $r\geq 1$, $x_i\in V_h^r$ and we can define the map $~\hat{\cdot}: {\mf g}_h(V_h^r) \rightarrow (V_h^r)^n$ by
\begin{equation}
\hat{\sf A}^h \coloneqq \sum_{i=1}^n({\sf A}^h x_i)\, {\bf e}_i\,,
\end{equation}
for all ${\sf A}^h \in {\mf g}_h$. By a standard result of mixed finite element theory (see, e.g., Corollary 2.3.1 in \cite{Boffi13}), for any ${\bf u} \in \Wh^s$ we have that ${\bf u}|_K \in ({\mc P}_s(K))^n$ for any $K\in {\mc T}_h$. Then, since $r \geq s$, for all ${\bf u} \in \Wh^s$, $u_i = {\bf u} \cdot {\bf e}_i \in V_h^r$. Moreover, we can verify by direct calculation, i.e.\ applying Definition~\ref{def:dLiehdiv}, that for all ${\sf X}^h_{\bf u}\in S_h^s(V_h^r)$, ${\sf X}^h_{\bf u}\, x_i = u_i$. Hence the restriction of the $~\hat{\cdot}~$ map to $S_h^s(V^r_h)$ is the map ${\sf X}^h_{\bf u} \mapsto {\bf u}$ which is a surjection from $S_h^s(V_h^r)$ to $\Wh^s$. On the other hand the map ${\bf u} \mapsto {\sf X}^h_{\bf u}$ from $\Wh^s$ to $S_h^s(V_h^r)$ is also a surjection by definition of $S_h^s(V_h^r)$, hence the result. 
\end{proof}

\begin{figure}
 \[
 \begin{minipage}[t]{0.2\linewidth}
 \begin{center}
  \begin{tikzcd}[column sep=-0.5em, row sep=3em]
  G\arrow{d}{\rho}  \\
  G(V) \arrow[dashed]{d}  &\subset& GL(V) \arrow{d}{\renewcommand{\arraystretch}{0.7}\begin{array}{l}\text{Restriction}\\ \hspace{1pt} \text{to } V_h\end{array}}\\
  G_h(V_h)  &\subset & GL(V_h)
  \end{tikzcd}
  \end{center}
  \end{minipage}\hspace{6em} %
  \begin{minipage}[t]{0.35\linewidth} %
  \vspace{-0.5em}
  \begin{center} %
  \begin{tikzcd}[column sep=-0.5em, row sep=3em]
  \Wh&\mathfrak{g}\arrow{d}{\sf L} \arrow[dashed]{l} \\
  &\mathfrak{g}(V) \arrow[dashed]{d}  &\subset& \mathfrak{gl}(V) \arrow{d}{\renewcommand{\arraystretch}{0.7}\begin{array}{l}\text{Restriction}\\
  \hspace{1pt} \text{to } V_h\end{array}}\\
  \arrow{uu}{\mathlarger{~\hat{\cdot}~}} S_h(V_h)  \subset&\mathfrak{g}_h(V_h)  &\subset & \mathfrak{gl}(V_h)
  \end{tikzcd}
 \end{center}
 \end{minipage}
 \]
\caption{Diagrams representing the discretisation of the group $G\coloneqq\mr{Diff}_\mr{vol}(\Omega)$ and its Lie algebra. The dashed arrows represent relations that hold only in an approximate sense.}
\label{fig:diagrams}
\end{figure}
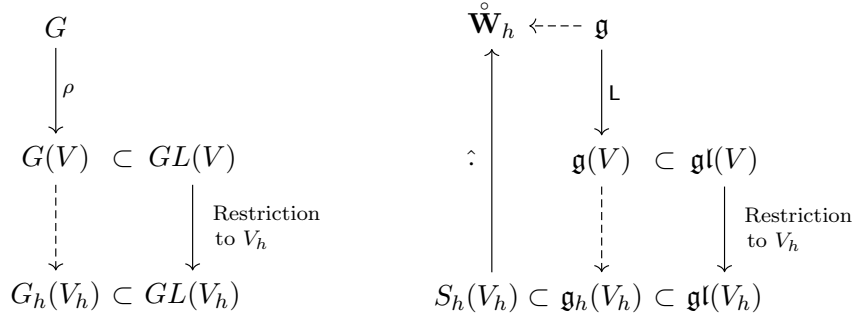

The different spaces introduced so far can be collected in the diagrams represented in Figure~\ref{fig:diagrams}. The definition of the $~\hat{\cdot}~$ map suggests the choice of a Lagrangian $l:\mathfrak{g}_h(V_h)\rightarrow \mathbb{R}$ given by
\begin{equation}
l({\sf A}^h) \coloneqq \frac{1}{2} \left(\hat{\sf A}^h,\hat{\sf A}^h \right)_{\Omega}\,,
\end{equation}
so that the restriction of $l$ to $S_h(V_h)$ is just the total kinetic energy, i.e.\ $l({\sf X}_{\bf u}^h) =  \|{\bf u}\|^2_{\Omega}/2$. 
Such a choice for the Lagrangian was already suggested in \cite{Pavlov1} as an alternative to the Lagrangian construction proposed therein based on dual grids.

In order to get a geometric picture analogous to the continuous case, we need a Lagrangian defined on the whole tangent space $TG_h(V_h)$ and approximating Equation~\eqref{eq:lag}. We achieve this by extending $l$ by right translations to get $L:TG_h(V_h)\rightarrow \mathbb{R}$ defined by
\begin{equation}\label{eq:disclag}
L(g^h, \dot{g}^h) \coloneqq l({\dot{g}}^h (g^{h})^{-1}) = \frac{1}{2} \int_\Omega \|\reallywidehat{{\dot{g}}^h (g^{h})^{-1}}  \|^2  \mr{vol}\, ,
\end{equation}
which is right invariant by construction. Analogously the space $S_h(V_h)$ can be extended by right translation to give a right invariant sub-bundle of $TG_h(V_h)$. To see how this is done, consider the right translation map $R_{g^h} : q^h \in G_h(V_h) \rightarrow q^hg^h \in G_h(V_h)$, and denote by $T_eR_{g^h} : {\sf A}^h \in {\mf g}_h(V_h) \rightarrow {\sf A}^h g^h \in T_{g^h}G_h(V_h)$ its tangent map at the identity (see \cite{Marsden10}, for more details). Then, the right invariant extension of $S_h(V_h)$ is obtained by collecting the spaces $T_eR_{g^h} S_h(V_h)$ for all $g^h\in G_h(V_h)$, where $T_eR_{g^h} S_h(V_h) \subset T_{g^h} G_h(V_h)$ is the space of linear operators which can be written as the composition ${\sf A}^h g^h$, for a given ${\sf A}^h\in S_h(V_h)$.

\begin{lemma} The Euler-Poincar\'e-d'Alembert equations relative to the right invariant Lagrangian $L$ with constraint $\dot{g}^h\in T_eR_{g^h} S_h(V_h)$ are given by: Find ${\sf A}^h\in S_h(V_h)$ such that
\begin{equation}\label{eq:epd}
\left(\der{}{t} \hat{\sf A}^h,\hat{\sf B}^h \right)_{\Omega} + \left (\hat{\sf A}^h,[\reallywidehat{{\sf A}^h, {\sf B}^h}] \right )_{\Omega} = 0,
\end{equation}
for all  ${\sf B}^h\in S_h(V_h)$, where $[{{\sf A}^h, {\sf B}^h}]\coloneqq {\sf A}^h {\sf B}^h - {{\sf B}^h {\sf A}^h}$ is the commutator of the linear operators ${\sf A}^h$ and ${\sf B}^h$.
\end{lemma}
\begin{proof}
The proof is completely analogous to that of Theorem 1 in \cite{Pavlov1}. In essence it is just an application of the reduction theorem (Theorem 13.5.3 in \cite{Marsden10}). In particular, note that proceeding as Equation~\eqref{eq:adlie} we easily get $\mr{ad}_{{\sf A}^h} {\sf B}^h = [{{\sf A}^h, {\sf B}^h}]$. Then, Equation~\eqref{eq:epd} is the finite-dimensional analogue of Equation~\eqref{eq:EPfluid}. The only difference is that we added a nonholonomic constraint on the velocity so that both the solution and the variations are constrained in the space $S_h(V_h)$. 
\end{proof}

\begin{theorem}\label{th:equivalence}
Under the hypotheses of Lemma~\ref{lem:iso}, the Euler-Poincar\'e-d'Alembert equations~\eqref{eq:epd} relative to the right invariant Lagrangian in Equation~\eqref{eq:disclag} with $S_h(V_h) = S^s_h(V^r_h)$ are equivalent to: Find $u \in \wh^s$ such that 
\begin{equation}\label{eq:disceuler}
\left (\dot{u} , v \right )_{\Omega} + \left ({\sf X}^h_{\bf u} u, v\right )_{\Omega} = 0,
\end{equation}
with $u = {\bf u} \cdot {\ed {\bf x}}$, for all ${v} \in \wh^s$.
\end{theorem}
\begin{proof}
First of all, by Lemma~\ref{lem:iso} for each ${\sf A}^h\in S_h^s(V_h^r)$ and ${\sf B}^h\in S_h^s(V_h^r)$ there is a unique ${\bf u}\in \Wh^s$ and ${\bf v}\in \Wh^s$ such that ${\sf A}^h=-{\sf X}^h_{\bf u}$ and  ${\sf B}^h=-{\sf X}^h_{\bf v}$.
We now examine the two terms in Equation~\eqref{eq:epd} separately. For the first term, we have
\begin{equation}
\left(\der{}{t} \hat{\sf A}^h,\hat{\sf B}^h \right)_{\Omega} = \left (\dot{u} , v \right )_{\Omega} \,,
\end{equation}
by definition of the one-form pairing in Equation~\eqref{eq:formpair}. As for the second term,  we have 
\begin{equation}\label{eq:firsteq}
\begin{aligned}
\left (\hat{\sf A}^h,[\reallywidehat{{\sf A}^h, {\sf B}^h}] \right )_{\Omega} &= -\left (\hat{\sf X}_{\bf u}^h,[\reallywidehat{{\sf X}_{\bf u}^h, {\sf X}_{\bf v}^h}] \right )_{\Omega}\\
&= - \sum_{i=1}^n\left (u_i,{\sf X}_{\bf u}^h v_i - {\sf X}_{\bf v}^h u_i  \right )_{\Omega}\\
&= -\sum_{i=1}^n\sum_K (u_i, {\sf L}_{\bf u}v_i - {\sf L}_{\bf v} u_i)_{K}+\sum_{i=1}^n\sum_{f\in {\mc F}^\circ} \Big( (\{u_i\},[v_i])_{f,\bf u}
- (\{u_i\},[u_i])_{f,\bf v}\Big)\\
&= -\sum_K ({\bf u}, [{\bf u}, {\bf v}])_{K}+\sum_{i=1}^n\sum_{f\in {\mc F}^\circ} \Big( (\{u_i\},[v_i])_{f,\bf u}
- (\{u_i\},[u_i])_{f,\bf v} \Big)\,. 
\end{aligned}
\end{equation}
where we used the identity $\sum_{i=1}^n({\sf L}_{\bf u}v_i - {\sf L}_{\bf v} u_i) {\bf e}_i = [{\bf u} , {\bf v}]$, which is a consequence of Equations~\eqref{eq:lie0} and \eqref{eq:ljbrac}. As for the facet integrals, we first notice that
\begin{equation}\label{eq:facet}
\sum_{i=1}^n\sum_{f\in {\mc F}^\circ} \Big((\{u_i\},[v_i])_{f,\bf u}
- (\{u_i\},[u_i])_{f,\bf v}\Big) =
 \sum_{f\in {\mc F}^\circ} \Big((\{u\},[v])_{f,\bf u}
- (\{u\},[u])_{f,\bf v}\Big)\, .
\end{equation}
Moreover, since $\bf u$  has continuous normal component on the mesh facets, we can substitute   
\begin{equation}\label{eq:facet1}
(\{u\},[v])_{f,\bf u} = (\{u\},[v])_{f,\{{\bf u}\}} = ({\bf n}_f\times\{{\bf u}\},{\{{\bf u}\}} \times [{\bf v}])_{f}\, ,
\end{equation}
where  for the last equality we have used the same reasoning as in the proof of Lemma~\ref{lem:alt}. Similarly, for the second term in Equation~\eqref{eq:facet}
we have
\begin{equation}\label{eq:facet2}
(\{u\},[u])_{f,\bf v} = ({\bf n}_f\times\{{\bf u}\},{\{{\bf v}\}}\times [{\bf u}])_{f} = -({\bf n}_f\times\{{\bf u}\},{[\bf u]}\times \{{\bf v}\})_{f} \,.  
\end{equation}
Therefore, Equation~\eqref{eq:facet} becomes
\begin{equation}\label{eq:facetn}
\begin{aligned}
\sum_{i=1}^n\sum_{f\in {\mc F}^\circ}\Big((\{u_i\},[v_i])_{f,\bf u}
- (\{u_i\},[u_i])_{f,\bf v}\Big) &=
 \sum_{f\in {\mc F}^\circ} ({\bf n}_f\times\{{\bf u}\}, \{{\bf u}\}\times [{\bf v}]+ {[\bf u]}\times \{{\bf v}\})_{f}\\
 &=  \sum_{f\in {\mc F}^\circ} ({\bf n}_f\times\{{\bf u}\}, [{\bf u}\times {\bf v}])_{f}
 \, .
\end{aligned}
\end{equation}
The last equality in Equation~\eqref{eq:facetn} follows from the linearity of the cross product, i.e.\
\begin{equation}\label{eq:cross}
\begin{aligned}
{\{{\bf u}\}} \times [{\bf v}] + {[\bf u]}\times \{{\bf v}\} &=  ({\bf u}^++{\bf u}^-)/2\times ({\bf v}^+-{\bf v}^-) + ({\bf u}^+-{\bf u}^-)\times ({\bf v}^++{\bf v}^-)/2 \\ 
&=  {\bf u}^+\times {\bf v}^+- {\bf u}^-\times{\bf v}^-=[{\bf u}\times{\bf v}]\, .
\end{aligned}
\end{equation}
Inserting Equation~\eqref{eq:facetn} into Equation~\eqref{eq:firsteq} yields
\begin{equation}
\left (\hat{\sf A}^h,[\reallywidehat{{\sf A}^h, {\sf B}^h}] \right )_{\Omega} = \sum_K ({\bf u}, {\bf{curl}}({\bf u}\times {\bf v}))_{K}+ \sum_{f\in {\mc F}^\circ} ({\bf n}_f\times\{{\bf u}\}, [{\bf u}\times {\bf v}])_{f}\, ,
\end{equation}
where we used the identity $[{\bf u},{\bf v}] = -{\bf{curl}}({\bf u}\times {\bf v})$ valid for any divergence-free vector fields ${\bf u}$ and ${\bf v}$. Recalling that $c_f=0$, comparison with Equation~\eqref{eq:clas2} concludes the proof.
\end{proof}

\begin{proposition}
The centred dicretisation in Equation~\eqref{eq:disceuler}, i.e.\ the case $c_f=0$, on a simplicial triangulation coincides with the discontinuous Galerkin discretisation with centred fluxes proposed in \cite{Guzman16}, which can be written as follows: Find $({\bf u},p) \in {\bf W}_h^s \times V_h^k $, where $k=s$ if ${\bf W}_h^s = \bf{RT}_s(\mc T_h)$, or $k=s-1$ if ${\bf W}_h^s = \bf{BDM}_s(\mc T_h)$, such that 
\begin{equation}\label{eq:guzman}
\left \{
\begin{array}{l}
\displaystyle(\dot{\bf u},{\bf v})_\Omega - \sum_K( {\bf u},{\bf u} \cdot \nabla {\bf v})_K + \sum_{f \in\mc F^\circ} ({\bf u} \cdot {\bf n}_f \{{\bf u}\},[{\bf v}])_f -\sum_K(p,\mr{div}\, {\bf v})_K = 0\, \\
\displaystyle \sum_K(\mr{div}\, {\bf u}, q)_K = 0\, 
\end{array}
\right .
\end{equation}
for all $({\bf v} ,q) \in {\bf W}_h^s \times V_h^k$, with ${\bf u} \cdot {\bf n}_{\partial \Omega} = 0$ on $\partial \Omega$, and 
$\int_{\Omega} p \, \ed x = 0$.
\end{proposition}
\begin{proof}
First, we observe that restricting Equation~\eqref{eq:guzman} on the subspace $\Wh^s$ yields the equivalent problem: Find ${\bf u} \in \Wh^s$ such that 
\begin{equation}\label{eq:guzman1}
(\dot{\bf u},{\bf v})_\Omega - \sum_K ( {\bf u},{\bf u} \cdot \nabla {\bf v})_K + \sum_{f \in\mc F^\circ} ({\bf u} \cdot {\bf n}_f \{{\bf u}\},[{\bf v}])_f  = 0\, ,
\end{equation}
for all ${\bf v} \in \Wh^s$. On the other hand, we can write Equation~\eqref{eq:disceuler} in the following form 
\begin{equation}\label{eq:Lie}
(\dot{\bf u},{\bf v})_\Omega + \sum_K ({\bf u}, \mr{ \bf{ curl}}({\bf u} \times{\bf v}))_K + \sum_{f \in\mc F^\circ} ({\bf n}_f\times \{{\bf u}\},[{\bf u}\times {\bf v} ])_f = 0\,.
\end{equation}
We rewrite the volume integrals in Equation~\eqref{eq:Lie} as follows,
\begin{equation}\label{eq:Lie1}
({\bf u}, \mr{ \bf{ curl}}({\bf u} \times{\bf v}))_K =  ({\bf u}, -{\bf u} \cdot \nabla {\bf v}+ {\bf v} \cdot \nabla {\bf u})_K = -({\bf u}, {\bf u} \cdot \nabla {\bf v})_K +(\nabla {\bf u}^2/2, {\bf v})_K\, ,
\end{equation}
for any $K \in \mc {T}_h$. We rewrite the facet integrals in Equation~\eqref{eq:Lie} as follows,
\begin{equation}\label{eq:Lie2}
\begin{aligned}
\sum_{f \in\mc F^\circ} ({\bf n}_f\times \{{\bf u}\},[{\bf u}\times {\bf v} ])_f &=\sum_{f \in\mc F^\circ} \Big ( ({\bf u} \cdot{\bf n}_f  \{{\bf u}\},[ {\bf v} ])_f - ({\bf v} \cdot{\bf n}_f  \{{\bf u}\},[ {\bf u} ])_f \Big )\\
&=\sum_{f \in\mc F^\circ} \Big ( ({\bf u} \cdot{\bf n}_f  \{{\bf u}\},[ {\bf v} ])_f - ({\bf v} \cdot{\bf n}_f ,[ {\bf u}^2/2 ])_f \Big )\\
&=\sum_{f \in\mc F^\circ} ({\bf u} \cdot{\bf n}_f  \{{\bf u}\},[ {\bf v} ])_f - \sum_{K} \int_K\mr{ div} ( {\bf v}\, ({\bf u}^2/2 ))\ed x \\
&=\sum_{f \in\mc F^\circ} ({\bf u} \cdot{\bf n}_f  \{{\bf u}\},[ {\bf v} ])_f - \sum_{K}({\bf v}, \nabla({\bf u}^2/2 ))_K\,,
\end{aligned}
\end{equation}
where we used Equation~\eqref{eq:facetn} for the equality in the first line.
Inserting Equation~\eqref{eq:Lie1} and \eqref{eq:Lie2} into Equation~\eqref{eq:Lie} gives the equivalence of the two algorithms.
\end{proof}

\subsection{Upwind scheme}\label{sec:upwind}

In this section, we study the upwind version of the scheme in Equation~\eqref{eq:disceuler}. In particular, we extend the discussion of the previous section to the upwind case, with the aim of preserving the main features of the variational derivation.

Including upwinding in the variational framework described above is not straightforward. This is because, if we define $S_h(V_h)$ by 
\begin{equation}
S_h(V_h)\coloneqq \{ {\sf X}_{\bf u}^h:V_h \rightarrow V_h \,~\text{with}~ c_f= \bar{c}_f({\bf u}) \, : \, {\bf u} \in \Wh\}\,,
\end{equation}
where $\bar{c}_f$ is a given function of $\bf u$, then $S_h(V_h)$ ceases to be a linear space, and furthermore $S_h(V_h)\nsubseteq \mathfrak{g}_h(V_h)$. This last issue can be solved easily by choosing a group larger than $G_h(V_h)$ as configuration space. More specifically, we need to allow for discrete diffeomorphisms which do not satisfy Equation~\eqref{eq:prop1}. We call the new configuration space $\tilde{G}_h(V_h)$ and we define it as follows.

\begin{definition} The Lie group $\tilde{G}_h(V_h)\subset GL(V_h)$ is the subgroup of the group of linear invertible operators from $V_h$ to itself, defined by the following property: for all $g^h\in \tilde{G}_h(V_h)$,
\begin{equation}\label{eq:prop1ex}
 g^h c = c  \qquad \forall\, c\in \mathbb{R}\,,
\end{equation}
where $c$ is considered as a constant function on $\Omega$.
\end{definition}

Denote with $\tilde{\mathfrak{g}}_h(V_h)$ the Lie algebra of $\tilde{G}_h(V_h)$. Then, by the same arguments as in the last section, it is easy to verify that, for any choice of $c_f$ in the definition of $S_h(V_h)$, $S_h(V_h)\subset \tilde{\mathfrak{g}}_h(V_h)$.
Next, we note that also Lemma~\ref{lem:iso} holds independently of the choice of $c_f$ in the definition of $S_h(V_h)$, and in particular, for fixed polynomial orders $r$ and $s$ with $r\geq1$ and $r\geq s$, if ${\sf A}^h \in S_h^s(V_h^r)$ then $\bf u\coloneqq \hat{\sf A}^h \in \Wh^s$. This observation leads to the following result.

\begin{proposition}[Upwinding]\label{prop:upwind} Under the hypotheses of Lemma~\ref{lem:iso}, the problem described by Equation~\eqref{eq:epd}, where $S_h(V_h) = S_h^s(V_h^r)$ is defined to be 
\begin{equation}\label{eq:newconstraint}
S_h^s(V_h^r)\coloneqq \{ {\sf X}_{\bf v}^h:V_h^r \rightarrow V_h^r \,~\text{with}~ c_f=  \bar{c}_f({\bf u}) \, : \, {\bf v} \in \Wh^s\}\,,
\end{equation}
with $\bf u\coloneqq \hat{\sf A}^h\in \Wh^s$, is equivalent to Equation~\eqref{eq:disceuler}, where ${\sf X}_{\bf u}^h$ is defined as in Definition~\ref{def:dLiehdiv} with $c_f=-{\bar{c}_f({\bf u})}$. {In particular, if we choose $\bar{c}_f({\bf u})=-{{\bf u} \cdot {\bf n}_f}/({2 |{\bf u} \cdot {\bf n}_f|})$, we obtain the upwind version of Equation~\eqref{eq:epd}}.
\end{proposition}
\begin{proof}
The proof is analogous to the one of Theorem~\ref{th:equivalence}.
\end{proof}

\begin{remark}
Note that the upwind version of the scheme in Equation~\eqref{eq:disceuler} cannot be derived directly from Hamilton's principle because the constraint space in Equation~\eqref{eq:newconstraint} is dependent on the solution of the problem. However, in the next section, we will see that the equivalence between Equation~\eqref{eq:disceuler} and Equation~\eqref{eq:epd}, ensured by Proposition~\ref{prop:upwind}, is enough to maintain the main properties of the centred discretisation (i.e.\ energy conservation and a discrete version of Kelvin's circulation theorem) also when upwinding is introduced.
\end{remark}

\subsection{Energy conservation and discrete Kelvin's circulation theorem}\label{sec:dkel}

A consequence of using a nonholonomic constraint in our approach is that our discretisation does not preserve all the Casimirs, i.e.\ the conserved quantities, of the original system. This is because the discrete bracket induced by the advection term in Equation~\eqref{eq:disceuler} is only a quasi-Poisson bracket (it is only skew-symmetric and does not satisfy the Jacobi identity). Nonetheless, because of the variational derivation of the equations of motion, we are  still able to ensure energy conservation and derive a discrete version of Kelvin's circulation theorem, with or without upwinding. Energy conservation is established in the following proposition.

\begin{proposition}[Energy conservation/$L^2$ stability] \label{prop:stab} The discrete system in Equation~\eqref{eq:disceuler}, with or without upwinding, satisfies conservation of the total kinetic energy, i.e.\ $\ed{l}/\ed t=0$.
\end{proposition}
\begin{proof}
This can be verified directly by setting $v=u$ in Equation~\eqref{eq:disceuler} and exploiting the antisymmetric structure of the equations.
\end{proof}

We now turn to derive a discrete Kelvin's circulation theorem. At the continuous level, the theorem states that for any closed loop $\Gamma$,
\begin{equation}\label{eq:kelvin}
\der{}{t} \int_{g_t \circ \Gamma} u = 0\,,
\end{equation}
where $g_t$ is the flow of ${\bf u}$, with $u= {\bf u}\cdot {\ed {\bf x}}$. 
In order to show in what sense Equation~\eqref{eq:kelvin} is verified at the discrete level, we will proceed as in \cite{Pavlov1} by replacing the integral in Equation~\eqref{eq:kelvin} by a pairing with currents, i.e.\ vector fields supported in a small region around a curve.  

A discrete current is a vector field ${\bf c}\in \Wh$ that is non-zero only on a closed loop of elements $K\in \mc T_h$ and such that the flux of ${\bf c}$ through adjacent elements of the loop is equal to 1. Clearly, with mesh refinement a discrete current approaches the notion of closed loop. In other words, if $\Gamma$ is a closed loop such that there exists a discrete current ${\bf c}$ with $\Gamma\in \overline{\mr{supp}({\bf c})}$, then
\begin{equation}
\langle u,{\bf c}\rangle \approx \int_{\Gamma} u \, .
\end{equation}
Let $g_t$ be the flow of $\bf u$ as defined by Equation~\eqref{eq:matvel} and \eqref{eq:spatial}. If both $g_t$ and a given discrete current $\bf c$ are smooth, then $\bf c$ is advected by $g_t$ via the map $ {\bf c} \mapsto (D g_t \, {\bf c}) \circ g_t^{-1}$, as it can be derived from Equation~\eqref{eq:adlie}.
Therefore, Kelvin's theorem can be reformulated using currents, by stating that
\begin{equation}\label{eq:kelvinc}
\der{}{t} \langle u, (D g_t \, {\bf c}) \circ g_t^{-1} \rangle  = 0\,,
\end{equation} 
for any smooth current $\bf c$. In the discrete setting $g_t$ can only be defined in weak sense, and in general it is not differentiable, because we only have ${\bf u} \in {\bm H}(\mr{div},\Omega)$. Nonetheless, we can still use the discrete flow map $g_t^h$ to reproduce an analogous of Equation~\eqref{eq:kelvinc}, as shown in the following proposition.

\begin{proposition}
Let ${\bf u} \in \Wh$, with $u={\bf u}\cdot {\ed {\bf x}}$, be a solution of the discrete Euler equations~\eqref{eq:disceuler}. Then, for all discrete currents ${\bf c}\in \Wh$
\begin{equation}\label{eq:kelvincd}
\der{}{t}\bigg|_{t=0}\langle u,  \reallywidehat{g^h_t{\sf X}^h_{\bf c} (g_t^h)^{-1}}  \rangle  = 0\,,
\end{equation} 
where $g_t^h\in G_h(V_h)$ {for the centred discretisation}, or $g_t^h\in \tilde{G}_h(V_h)$ {for the upwind discretisation}, is the discrete flow of $-{\sf X}_{\bf u}^h$, i.e.\ it satisfies $\dot{g}^h_t =- {\sf X}_{\bf u}^h g^h_t$ and $g^h_0 = e$. The result extends to any time $t$ by appropriately translating the definition of $g_t^h$ in time.
\end{proposition}
\begin{proof}
{For the centred discretisation}, let $q_t^h\in G_h(V_h)$ be the discrete flow of $-{\sf X}^h_{\bf c}$ (or $q_t^h\in \tilde{G}_h(V_h)$, {for the upwind discretisation}), then
\begin{equation}
[{\sf X}^h_{\bf u},{\sf X}^h_{\bf c}]=\mr{ad}_{{\sf X}^h_{\bf u}}{\sf X}^h_{\bf c} = \der{}{t}\bigg|_{t=0}
\der{}{s}\bigg|_{s=0} g_t^h q_s^h (g_t^h)^{-1}  = -\der{}{t}\bigg|_{t=0} g^h_t {\sf X}^h_{\bf c} (g_t^h)^{-1}\,.
\end{equation}
Finally, by Equation~\eqref{eq:epd}, we have
\begin{equation}
\der{}{t}\bigg|_{t=0} \langle u,  \reallywidehat{g^h_t{\sf X}^h_{\bf c} (g_t^h)^{-1}}  \rangle  = \langle \dot{u}, {\bf c} \rangle|_{t=0} -\langle {u}, [\reallywidehat{{\sf X}^h_{\bf u},{\sf X}^h_{\bf c}}] \rangle|_{t=0} = 0
\,.
\end{equation} 
\end{proof}
\section{Convergence analysis} \label{sec:conv}
In this section we analyse the convergence properties of the semi-discrete system in Equation~\eqref{eq:disceuler}. In particular, as the centred discretisation, i.e.\ the case $c_f =0$, was studied in \cite{Guzman16}, we will concentrate on the upwind formulation, i.e.\ the case $c_f = {\bf u} \cdot {\bf n}_f/ (2 |{\bf u} \cdot {\bf n}_f|)$. For simplicity, we continue to restrict to the case $n=3$, although the results of this section extend to $n=2$. We keep the dependence on $n$ explicit when needed. 

We start with characterising the space ${\bf W}_h$ and its approximation properties. We will assume that $\mc T_h$ is quasi-uniform and shape regular and that $\Wh$ is constructed by applying Equation~\eqref{eq:W}, by setting ${\bf W}_h = {\bf{RT}}_{s}(\mc T_h)$ or ${\bf W}_h = {\bf{BDM}}_{s}(\mc T_h)$, i.e.\ the Raviart-Thomas or Brezzi-Douglas-Marini finite element spaces of order $s$ on ${\mc T}_h$. Then, we say that ${\bf W}_h$ and $\Wh$ are polynomial spaces of order $s$. In particular, as shown in \cite{Boffi13} (or in \cite{Arnold06} in the context of FEEC), there exists a projection operator ${\bf P}_h$ of smooth vector fields onto ${\bf W}_h$ such that
\begin{equation}\label{eq:est1}
\|{\bf u} - {\bf P}_h {\bf u}\|_{\Omega} \leq C h^{r}|{\bf u}|_{{\bm H}^{r}(\Omega)} \, ,
\end{equation}
for any vector field ${\bf u}\in {\bm H}^{r}(\Omega)$ and $1\leq r \leq s+1$. 

From now on we will denote by ${\bf u}$ the exact solution of the Euler equation, and by ${\bf u}_h\in \Wh$ the discrete solution obtained by solving Equation~\eqref{eq:disceuler}. The estimate for the centred flux scheme in \cite{Guzman16} was derived as follows. First, the error is decomposed using the triangle inequality
\begin{equation}
\|{\bf u} - {\bf u}_h\|_{\Omega}  \leq 
\|{\bf u} - {\bf P}_h {\bf u}\|_{\Omega} +
\|{\bf P}_h {\bf u} - {\bf u}_h\|_{\Omega} \,.
\end{equation}
The first term on the right-hand side is the approximation error, therefore for the convergence estimate is sufficient to obtain a bound on ${\bm \gamma}_h\coloneqq {\bf P}_h {\bf u} - {\bf u}_h$. We collect the results of \cite{Guzman16} in the following Lemma.

\begin{lemma}[Theorem  2.1 in \cite{Guzman16}] \label{lem:central} The error ${\bm \gamma}_h$ satisfies the following inequality
\begin{equation}\label{eq:derest0}
\der{}{t} \|{\bm \gamma}_h\|^2_\Omega \leq C_0 h^{2s}  + C_1 h^{2s+2} + C_2 \|{\bm \gamma}_h\|^2_\Omega\, ,
\end{equation}
where $C_0,C_1,C_2>0$ are constants depending on $\|{\bf u}\|_{{\bm W}^{1,\infty}(\Omega)}$, $\|{\bf u}\|_{{\bm H}^{r+1}(\Omega)}$, $\|\dot{\bf u}\|_{{\bm H}^{r+1}(\Omega)}$, but independent of $h$. 
Then, if the norms above are uniformly bounded for $ t\in [0,T]$, the following estimate holds
\begin{equation}\label{eq:errorc}
\|{\bf u} - \bf{u}_h \|_\Omega \leq C h^s \, ,
\end{equation}
for an appropriate constant $C>0$ independent of $h$, and for all $t\in [0,T]$.
\end{lemma}

When upwinding is introduced, i.e.\ for the case $c_f ={\bf u} \cdot {\bf n}_f/ (2 |{\bf u} \cdot {\bf n}_f|)$, we just need to add the upwind contribution to the left-hand side of Equation~\eqref{eq:derest0} and proceed with the estimate. In particular, from the proof of Theorem 2.1 in \cite{Guzman16}, it can be easily verified  that, assuming the exact solution to be continuous, Equation~\eqref{eq:derest0} needs to be replaced by
\begin{equation}\label{eq:derest1}
\der{}{t}\| {\bm \gamma}_h\|_{\Omega}^2 - \sum_{f\in {\mc F}^\circ}(c_f\, {\bf n}_f\times[{\bf u}_h],[{\bf u}_h\times{\bm\gamma}_h])_f  \leq C_0 h^{2s}+ C_1 h^{2s+2} + C_2 \|{\bm \gamma}_h\|_{\Omega}^2 \,.
\end{equation} 
Proceeding as in Equation~\eqref{eq:firsteq}, we rewrite the upwind term as follows,
\begin{equation}\label{eq:upest}
\begin{aligned}
\sum_{f\in {\mc F}^\circ}(c_f\, {\bf n}_f\times[{\bf u}_h],[{\bf u}_h\times{\bm\gamma}_h])_f &= \sum_{f\in {\mc F}^\circ}(c_f\, {\bf n}_f\times[{\bf u}_h], \{{\bf u}_h\}\times[{\bm\gamma}_h]+[{\bf u}_h]\times\{{\bm\gamma}_h\})_f\\
&= \sum_{f\in {\mc F}^\circ} \Big( ([{\bf u}_h],[{\bm\gamma}_h])_{f,c_f {\bf u}_h}-([{\bf u}_h],[{\bf u}_h])_{f,c_f {\bm\gamma}_h} \Big)\\
&= \sum_{f\in {\mc F}^\circ} \Big( ([{\bf P}_h{\bf u } ],[{\bm\gamma}_h])_{f,c_f {\bf u}_h}- ([{\bm \gamma}_h],[{\bm\gamma}_h])_{f,c_f {\bf u}_h} -([{\bf u}_h],[{\bf u}_h])_{f,c_f {\bm\gamma}_h}\Big) \,.
\end{aligned}
\end{equation}
Reinserting Equation~\eqref{eq:upest} into Equation~\eqref{eq:derest1} gives
\begin{equation}\label{eq:derest2}
\begin{aligned}
   \der{}{t}\| {\bm \gamma}_h\|_{\Omega}^2  + \sum_{f\in {\mc F}^\circ} ([{\bm \gamma}_h],[{\bm\gamma}_h])_{f,c_f {\bf u}_h}  \leq&~ C_0 h^{2s}+C_1 h^{2s+2}+ C_2 \|{\bm \gamma}_h\|_{\Omega}^2\\&~
  + \sum_{f\in {\mc F}^\circ}([{\bf P}_h{\bf u } ],[{\bm\gamma}_h])_{f,c_f {\bf u}_h}-\sum_{f\in {\mc F}^\circ}([{\bf u}_h],[{\bf u}_h])_{f,c_f {\bm\gamma}_h} \, ,
\end{aligned}
\end{equation}
where the facet integrals on the left-hand side are always non-negative due to the definition of $c_f$. Define
\begin{equation}
I_1^{up} \coloneqq\bigg| \sum_{f\in {\mc F}^\circ}([{\bf P}_h{\bf u } ],[{\bm\gamma}_h])_{f,c_f {\bf u}_h}\bigg|\, ,\qquad I_2^{up} \coloneqq\bigg|\sum_{f\in {\mc F}^\circ}([{\bf u}_h],[{\bf u}_h])_{f,c_f {\bm\gamma}_h}\bigg| \, ,
\end{equation}
then
\begin{equation}\label{eq:derest3}
   \der{}{t}\| {\bm \gamma}_h\|_{\Omega}^2 + \sum_{f\in {\mc F}^\circ} ([{\bm \gamma}_h],[{\bm\gamma}_h])_{f,c_f {\bf u}_h} \leq C_0 h ^{2s}+ C_1 h^{2s+2}+ C_2 \|{\bm \gamma}_h\|_{\Omega}^2+I_1^{up}+I_2^{up} \, .
\end{equation}

\paragraph{Estimate for $I_1^{up}$}
Since $c_f {\bf u}_h\cdot {\bf n}_f = |{\bf u}_h \cdot {\bf n}_f|/2$, the following estimate holds,
\begin{equation}\label{eq:I1_est}
\begin{aligned}
I_1^{up}= \bigg|\sum_{f\in {\mc F}^\circ}([{\bf P}_h{\bf u } ],[{\bm\gamma}_h])_{f,c_f {\bf u}_h}\bigg| &\leq \frac{1}{2}\bigg| \sum_{f\in {\mc F}^\circ}([{\bf P}_h{\bf u } ],[{\bf P}_h{\bf u }])_{f,c_f {\bf u}_h} +  \sum_{f\in {\mc F}^\circ}([{\bm\gamma}_h ],[{\bm\gamma}_h])_{f,c_f {\bf u}_h}\bigg|\\
&\leq \frac{1}{2}\sum_{f\in {\mc F}^\circ}([{\bf P}_h{\bf u } -{\bf u} ],[{\bf P}_h{\bf u }-{\bf u}])_{f,c_f {\bf u}_h} + \frac{1}{2} \sum_{f\in {\mc F}^\circ}([{\bm\gamma}_h ],[{\bm\gamma}_h])_{f,c_f {\bf u}_h}\\
&\leq C_{tr} h^{-1} \|{\bf u}_h\|_{{\bm L}^\infty(\Omega)} \|{\bf P}_h{\bf u }-{\bf u}\|^2_\Omega +\frac{1}{2}\sum_{f\in {\mc F}^\circ}([{\bm\gamma}_h ],[{\bm\gamma}_h])_{f,c_f {\bf u}_h}\,,
\end{aligned}
\end{equation}
where we used the trace inequality $\|\cdot\|^2_{\partial K} \leq C_{tr} h^{-1} \|\cdot\|^2_{K}$, with $C_{tr}>0$ and independent of $h$.
We estimate the ${\bm L}^\infty$ norm of ${\bf u}_h$ as follows,
\begin{equation}\label{eq:uinf_est}
 \| {\bf u}_h\|_{{\bm L}^\infty(\Omega)}\leq \| {\bf u}\|_{{\bm L}^\infty(\Omega)} + C_{inv} h^{-n/2}(\|{\bf  u} - {\bf P}_h {\bf u}\|_{\Omega} +  \|{\bm \gamma}_h\|_{\Omega})\, ,
\end{equation}
where we used the inverse inequality $\|\cdot\|_{{\bm L}^{\infty}(\Omega)}\leq C_{inv} h^{-n/2}\|\cdot\|_\Omega$, with $C_{inv}>0$ and independent of $h$. Inserting Equation~\eqref{eq:uinf_est} into Equation~\eqref{eq:I1_est},and applying Young's inequality to the term involving $\|{\bm \gamma}_h\|_\Omega$, 
we obtain 
\begin{equation}\label{eq:I1}
I_1^{up}\leq C_3  (h^{2s+1}+h^{4s+2-n} + h^{3s+2-n/2}) +\frac{1}{2}  \|{\bm \gamma}_h\|^2_\Omega+\frac{1}{2}\sum_{f\in {\mc F}^\circ}([{\bm\gamma}_h ],[{\bm\gamma}_h])_{f,c_f {\bf u}_h}\,,
\end{equation}
for an appropriate constant $C_3>0$ independent of $h$.
\paragraph{Estimate for $I_2^{up}$} Proceeding as for $I_1^{up}$, we obtain
\begin{equation} \label{eq:I2}
\begin{aligned}
I_2^{up} =\bigg|\sum_{f\in {\mc F}^\circ}([{\bf u}_h],[{\bf u}_h])_{f,c_f {\bm\gamma}_h}\bigg|&\leq 2\bigg|\sum_{f\in {\mc F}^\circ}([{\bm \gamma}_h],[{\bm \gamma}_h])_{f,c_f {\bm\gamma}_h}\bigg| + 2\bigg|\sum_{f\in {\mc F}^\circ}([{\bf P}_h {\bf u} - {\bf u}],[{\bf P}_h {\bf u} - {\bf u}])_{f,c_f {\bm\gamma}_h}\bigg|\\
&\leq 4\,C_{tr}\,C_{inv}\,h^{-1-n/2}(\|{\bm \gamma}_h\|^3_\Omega +  \|{\bm \gamma}_h\|_{\Omega} \|{\bf P}_h{\bf u }-{\bf u}\|^2_\Omega)\\
&\leq C_{4}h^{-1-n/2}\|{\bm \gamma}_h\|^3_\Omega + 
C_{4}  h^{4s+2-n} + \frac{1}{2} \|{\bm \gamma}_h\|^2_\Omega\,,
\end{aligned}
\end{equation}
for an appropriate constant $C_4>0$ independent of $h$.
\paragraph{Final estimate} Assuming $s\geq 1$ we can neglect higher order terms in $h$ in the estimates above. In particular, inserting Equation~\eqref{eq:I1} and \eqref{eq:I2} into Equation~\eqref{eq:derest3}, we obtain  
\begin{equation}\label{eq:derest4}
  \der{}{t}\| {\bm \gamma}_h\|_{\Omega}^2 -\|{\bm \gamma}_h\|^2_\Omega + \sum_{f\in {\mc F}^\circ} ([{\bm \gamma}_h],[{\bm\gamma}_h])_{f,c_f {\bf u}_h} \leq C_5 h^{2s}+ C_5h^{-1-n/2}\|{\bm \gamma}_h\|^3_\Omega \, ,
\end{equation}
for a constant $C_5>0$ independent of $h$. We multiply both sides of Equation~\eqref{eq:derest4} by $e^{-t}$, and we let $q(t)\coloneqq e^{-t}\|{\bm \gamma}_h\|_{\Omega}^2 $, so we obtain 
\begin{equation}\label{eq:derest5}
  \der{}{t} q(t) + e^{-t}\sum_{f\in {\mc F}^\circ} ([{\bm \gamma}_h],[{\bm\gamma}_h])_{f,c_f {\bf u}_h} \leq e^{-t}C_5 h^{2s}+ e^{t/2} C_5 h^{-1-n/2}q(t)^{3/2} \, .
\end{equation}
Integrating from $0$ to $T$ gives
\begin{equation}\label{eq:derest6}
\begin{aligned}
  q(T) &\leq q(0) + (1-e^{-T}) C_5 h^{2s}+ C_5 h^{-1-n/2}\int_0^Te^{s/2}q(s)^{3/2} \ed s \\
  & \leq  C_6 h^{2s}+ C_6 h^{-1-n/2}\int_0^Te^{s/2}q(s)^{3/2} \ed s
  \,,
\end{aligned}
\end{equation}
where $C_6$ depends on $\|{\bf u}|_{t=0}\|_{{\bm H}^{r+1}(\Omega)}$ but not on $h$.
Then, the nonlinear generalisation of the Gronwall's inequality in \cite{Butler71} yields
\begin{equation}
q(T) \leq h^{2s} \left( \frac{1}{C_6^{1/2} } - 2C_6h^{s-1-n/2}(e^{T/2}-1) \right)^{-2}\,,
\end{equation}
with the condition
\begin{equation}
h^{s-1-n/2}< \frac{1}{2 C_6^{3/2} (e^{T/2}-1)} \,.
\end{equation}
If $s>1+n/2$ such a condition can be verified for any $T$, if $h$ is sufficiently small. Moreover, for 
\begin{equation}
h^{s-1-n/2}< \frac{1}{4 C_6^{3/2} (e^{T/2}-1)} \,,
\end{equation}
we get 
\begin{equation}
q(T) \leq  4C_6 h^{2s} \,,
\end{equation}
which proves the following theorem.

\begin{theorem}\label{th:conv} Under the same assumptions on $\bf u$ as in Lemma~\ref{lem:central}, 
the discrete solution ${\bf u}_h\in {\bm W}_h^s$, with $s>1+n/2$, obtained by solving the system in Equation~\eqref{eq:disceuler} with $c_f = {\bf u} \cdot {\bf n}_f/ (2 |{\bf u} \cdot {\bf n}_f|)$, satisfies
\begin{equation}
\|{\bf u} - {\bf u}_h\|_{\Omega} \leq C h^s\, ,
\end{equation} 
with $C>0$ dependent $T$, $\|{\bf u}|_{t=0}\|_{{\bm H}^{s+1}(\Omega)}$, $\|{\bf u}\|_{H^1([0,T],{\bm H}^{s+1}(\Omega))}$, $\|{\bf u}\|_{{\bm W}^{1,\infty}([0,T]\times \Omega)}$, but not on $h$.
\end{theorem}

\begin{remark}
Theorem~\ref{th:conv} gives a sub-optimal rate of convergence for the upwind scheme, i.e.\ $c_f={{\bf u}_h \cdot {\bf n}_f}/({2|{\bf u}_h \cdot {\bf n}_f|})$, only for the case $s>1+n/2$. This result is rather unsatisfactory, since the numerical tests in the next section suggest that the scheme converges with the optimal rate $s+1$ for $s\geq 1$. Unfortunately, we were not able to prove this result analytically.
\end{remark}

\section{Numerical tests} \label{sec:num}
We now show some numerical results illustrating our analytical results.  All the tests presented here are computed in the domain $\Omega = [0,2\pi] \times [0,2\pi]$ with spatial dimension $n=2$ and global Cartesian coordinates $(x_1,x_2)$. 
We combine the semi-discrete scheme introduced in the previous section with the implicit midpoint time discretisation, with fixed time step $\Delta t$. 
More specifically, let $t^n \coloneqq n \Delta t$ for $n \geq 0$, and $u^n = {\bf u}^n \cdot \ed {\bf x} \coloneqq  u|_{t=t^n}$, then the fully discrete scheme is defined by: Given $u^n \in \wh$, find $u^{n+1}\in \wh$ such that
\begin{equation}\label{eq:disceulert}
\frac{1}{\Delta t}\left ( {u}^{n+1}-u^n , v \right )_{\Omega} + \left ({\sf X}^h_{\frac{{\bf u}^{n}+{\bf u}^{n+1}}{2 }} \frac{u^{n}+u^{n+1}}{2}, v\right )_{\Omega} = 0\,,
\end{equation}
for all $v \in \wh$. It is well known that the implicit midpoint rule preserves  the quadratic invariants of the semi-discrete system exactly. In our case kinetic energy is conserved via Proposition~\ref{prop:stab}. The tests of this section were performed using the Firedrake software suite \cite{Rathgeber2016}, which
allows for symbolic implementation of finite element problems of mixed type. The nonlinear system was solved using Newton's method, with LU factorisation for the inner linear system implemented using the PETSc library \cite{petsc-user-ref,petsc-efficient,Dalcin2011,Chaco95}.

We start by  examining the convergence rates with respect to a given manufactured solution. The manufactured solution solves a modified system of equations where an appropriate forcing is introduced so that we can study the approximation properties of our discretisation (although we cannot make conclusive remarks on stability). We pick as manufactured solution the Taylor-Green vortex \cite{Guzman16},
\begin{equation}
{\bf  u}(t,x_1,x_2) = \sin(x_1)\cos(x_2)e^{-2t/\sigma} {\bf e}_1 -\cos(x_1)\sin(x_2)e^{-2t/\sigma} {\bf e}_2 \, ,
\end{equation}
with $\Omega = [0,2\pi]\times [0,2\pi]$, $\sigma=100$ and $t\in [0,1]$. 
Table~\ref{tab:centup} shows the $L^2$-error and the convergence rates for both the centred and upwind scheme, for the case
${\bf W}_ h = {\bf{RT}}_s(\mc T_h)$, and time step $\Delta t = 10^{-2}$. The centred scheme was already analysed in \cite{Guzman16}, where it was pointed out that without upwinding the rate of convergence $s$ is sharp for $s=1$, whereas convergence rates higher than $s$ can be observed for $s=0$ and $s=2$. The form of upwinding proposed in this paper is different from the one of \cite{Guzman16}. However, it gives the same behaviour in terms of order of convergence. In particular, we notice that introducing upwinding yields the optimal convergence rate $s+1$ for $s\geq 1$.

\begin{table}
\begin{center}
\renewcommand{\arraystretch}{1.2} 
\begin{tabular}{ c c|c c|c c}
\hline \hline
\multirow{2}{*}{$s$} & \multirow{2}{*}{$h$} & \multicolumn{2}{|c|}{centred}& \multicolumn{2}{|c}{upwind}\\
& & error
& order& error
& order\\
\hline
\multirow{4}{*}{0} & 7.40e-1 & 2.84e-1  & & 4.01e-1&  \\
& 3.70e-1& 1.42e-1
 & 1.00 &  2.24e-1 & 0.84\\
& 2.47e-1& 9.50e-2
 & 1.00 & 1.58e-1
  & 0.87 \\
& 1.85e-1&  7.12e-2& 1.00&   1.22e-1& 0.89 \\
\hline
\multirow{4}{*}{1} & 7.40e-1 & 1.42e-1 & & 2.15e-2&  \\
& 3.70e-1 & 
 7.13e-2& 0.99 &
  5.38e-3
   & 1.99\\
& 2.47-1 & 
4.76e-2
& 1.00 & 2.39e-3
& 2.00\\
& 1.85e-1 & 3.57e-2& 1.00&   1.35e-3& 2.00\\
\hline
\multirow{4}{*}{2} & 7.40e-1 & 1.81e-3
& & 7.61e-4
 &   \\
& 3.70e-1&  2.09e-4  
& 3.11 & 9.02e-5
 & 3.08   \\
& 2.47e-1& 6.28e-5 & 2.97 &  2.59e-5 & 3.08  \\
& 1.85e-1& 
 2.69e-5& 2.96 & 1.07e-5&  3.07 \\
\hline
\end{tabular}
\end{center}
\caption{Comparison between centred ($c_f=0$) and upwind ($c_f = \frac{{\bf u}\cdot{\bf n}_f}{2|{\bf u}\cdot{\bf n}_f|}$) scheme in terms of the error $\|{\bf{u}}-\bf{u}_h\|_{\Omega}$ and the order of convergence, for the Taylor-Green vortex test case and ${\bf W}_h={\bf{RT}}_s(\mc T_h)$.}\label{tab:centup}
\end{table} 

In the following we will refer to the scheme of \cite{Guzman16} as the GSS scheme (by the initial of the authors), whereas we refer to the scheme proposed in this paper as the Lie derivative scheme. 

We now consider a double shear problem \cite{liu00,Guzman16} obtained by setting the initial velocity ${\bf u} = u_1{\bf e}_1 + u_2 {\bf e}_2$ as follows
\begin{equation}
u_1|_{t=0} \coloneqq \left \{ 
\begin{array}{ll} 
\tanh((x_2-\pi/2)/\rho) \quad & x_2 \leq \pi \\
\tanh(( 3\pi/2 -x_2)/\rho) \quad & x_2 > \pi 
\end{array}
\right. \, , \qquad
u_2|_{t=0} \coloneqq \delta \sin (x_1)\, ,
\end{equation}
where $\rho = \pi/15$, $\delta = 0.05$, and with periodic boundary conditions on $\Omega$. We use $\Delta t = 8/200$ as in the references above to integrate the solution in time. In Figures~\ref{fig:ensnap} to \ref{fig:ensnap2}, we plot the vorticity $\mr{rot}\, {\bf u}$ at $t=8$, for the centred scheme and the upwind GSS and Lie derivative schemes. The centred scheme produces oscillations in the vorticity field, as it was already observed in \cite{Guzman16}. The presence of oscillations is reflected in the growth in enstrophy $Z(t)\coloneqq \int_{\Omega } (\mr{rot}\, {\bf u})^2\, \ed {\bf x} $, which accumulates at small scales, see Figure~\ref{fig:enshist1}. As a matter of fact, unfortunately, our variational derivation does not imply enstrophy conservation, even though this is an invariant of the continuous system. On the other hand, the energy $K(t) \coloneqq  \int_{\Omega } \|{\bf u}\|^2\, \ed {\bf x}$ is exactly conserved by the scheme.
Furthermore, we observe from Figure~\ref{fig:enshist1} that with $h$ refinement the enstrophy at the final time does not seem to converge to its initial value, when $s=1$; whereas we do observe convergence for $s=2$. This is expected since the sub-optimal error estimates provided in \cite{Guzman16} are sharp for $s=1$, so we can only guarantee convergence of the vorticity in $L^2$ for $s\geq 2$ via an inverse estimate.

The vorticity fields for the upwind GSS and Lie derivative schemes are similar, and provide much better agreement to the reference solution \cite{liu00}. The enstrophy time evolution in Figure~\ref{fig:enshist2} indicates that the introduction of upwinding induces dissipation in enstrophy.  
However, the upwind Lie derivative schemes still preserves energy exactly, whereas the upwind GSS scheme also dissipates energy.

\begin{figure}
\noindent
\subfigure[$h=0.185$, $s=1$ ($\|\mr{rot}\, {\bf u}\|_{L^\infty(\Omega)}\approx 14.05$)]{\includegraphics[width=0.5\linewidth, trim= 80 50 20 50 ,clip]{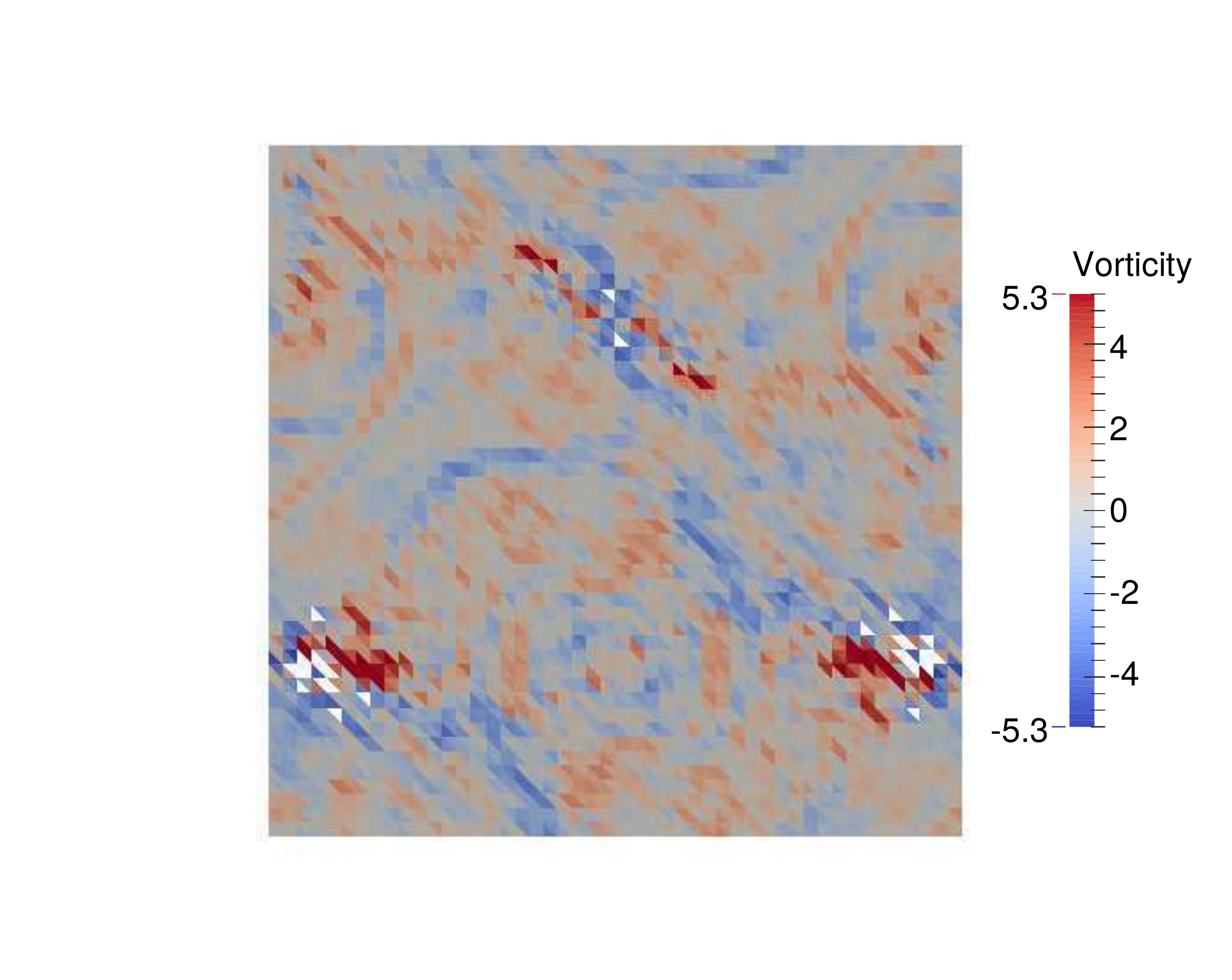}}
\subfigure[$h=0.185$, $s=2$ ($\|\mr{rot}\, {\bf u}\|_{L^\infty(\Omega)}\approx 11.63$)]{\includegraphics[width=0.5\linewidth, trim= 80 50 20 50 ,clip]{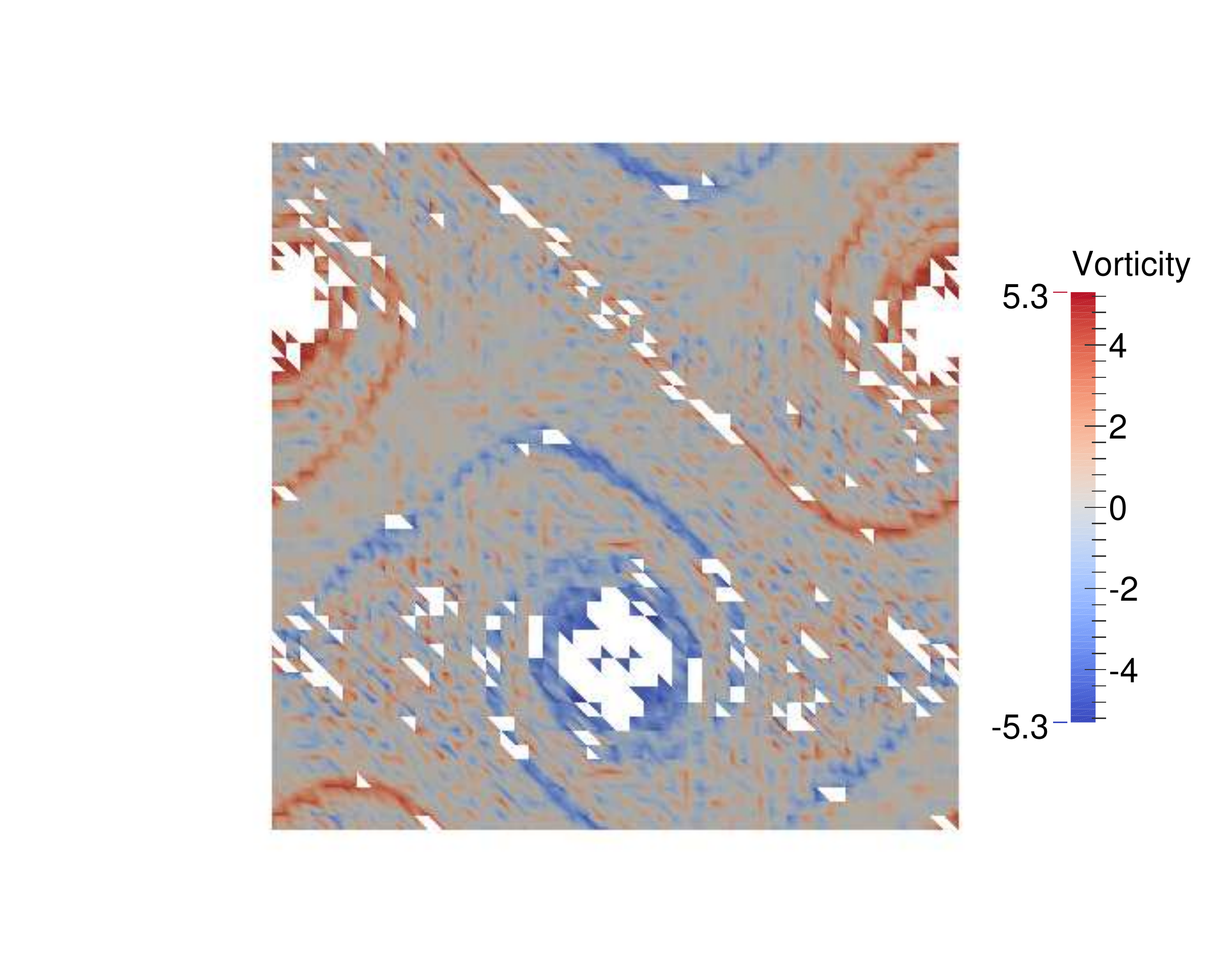}}
\subfigure[$h=0.135$, $s=1$ ($\|\mr{rot}\, {\bf u}\|_{L^\infty(\Omega)}\approx 23.86$)]{\includegraphics[width=0.5\linewidth, trim= 80 50 20 50 ,clip]{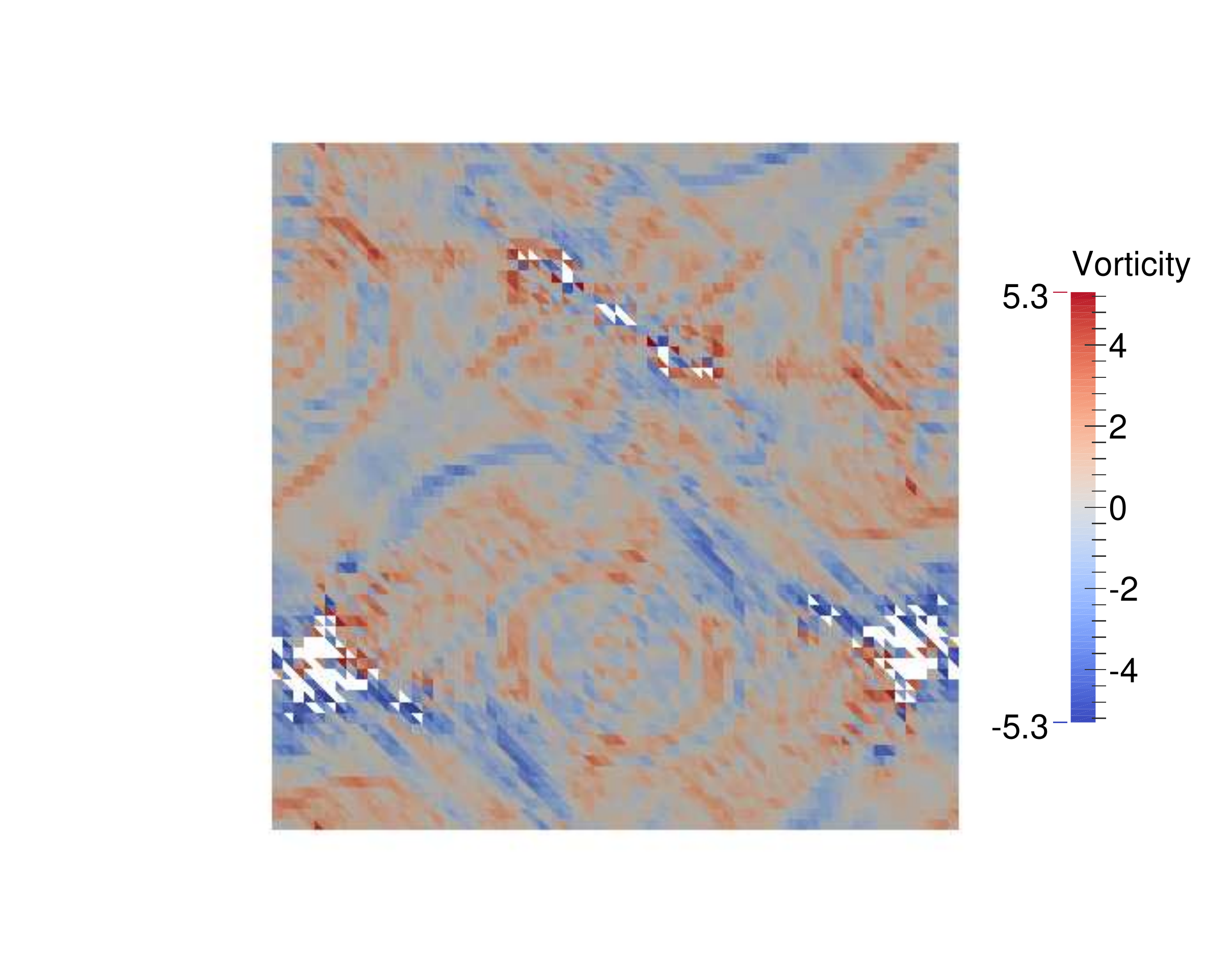}}
\subfigure[$h=0.139$, $s=2$ ($\|\mr{rot}\, {\bf u}\|_{L^\infty(\Omega)}\approx 11.41$)]{\includegraphics[width=0.5\linewidth, trim= 80 50 20 50 ,clip]{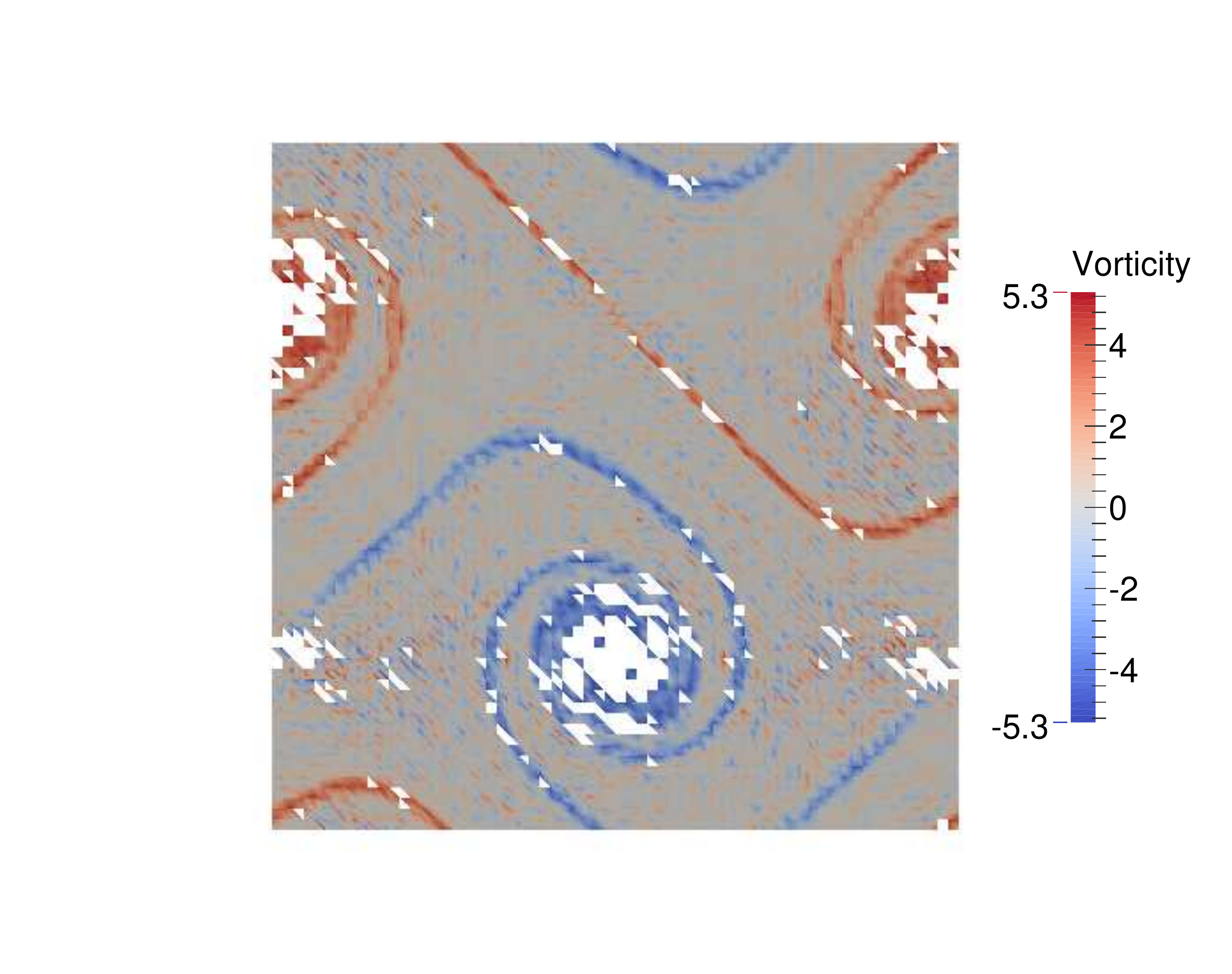}}
\subfigure[$h=0.104$, $s=1$ ($\|\mr{rot}\, {\bf u}\|_{L^\infty(\Omega)}\approx 29.93$)]{\includegraphics[width=0.5\linewidth, trim= 80 50 20 50 ,clip]{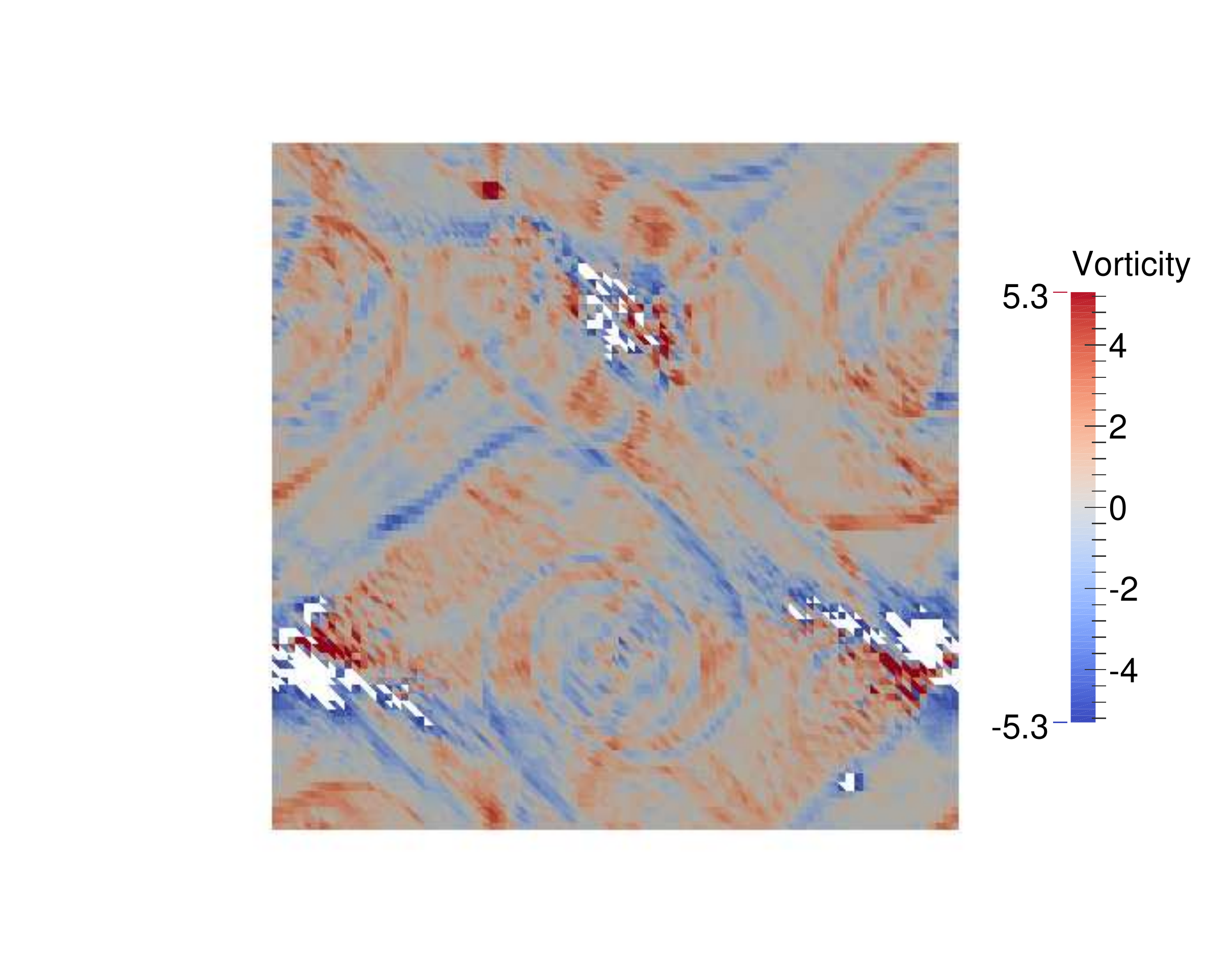}}
\subfigure[$h=0.104$, $s=2$ ($\|\mr{rot}\, {\bf u}\|_{L^\infty(\Omega)}\approx 17.74$)]{\includegraphics[width=0.5\linewidth, trim= 80 50 20 50 ,clip]{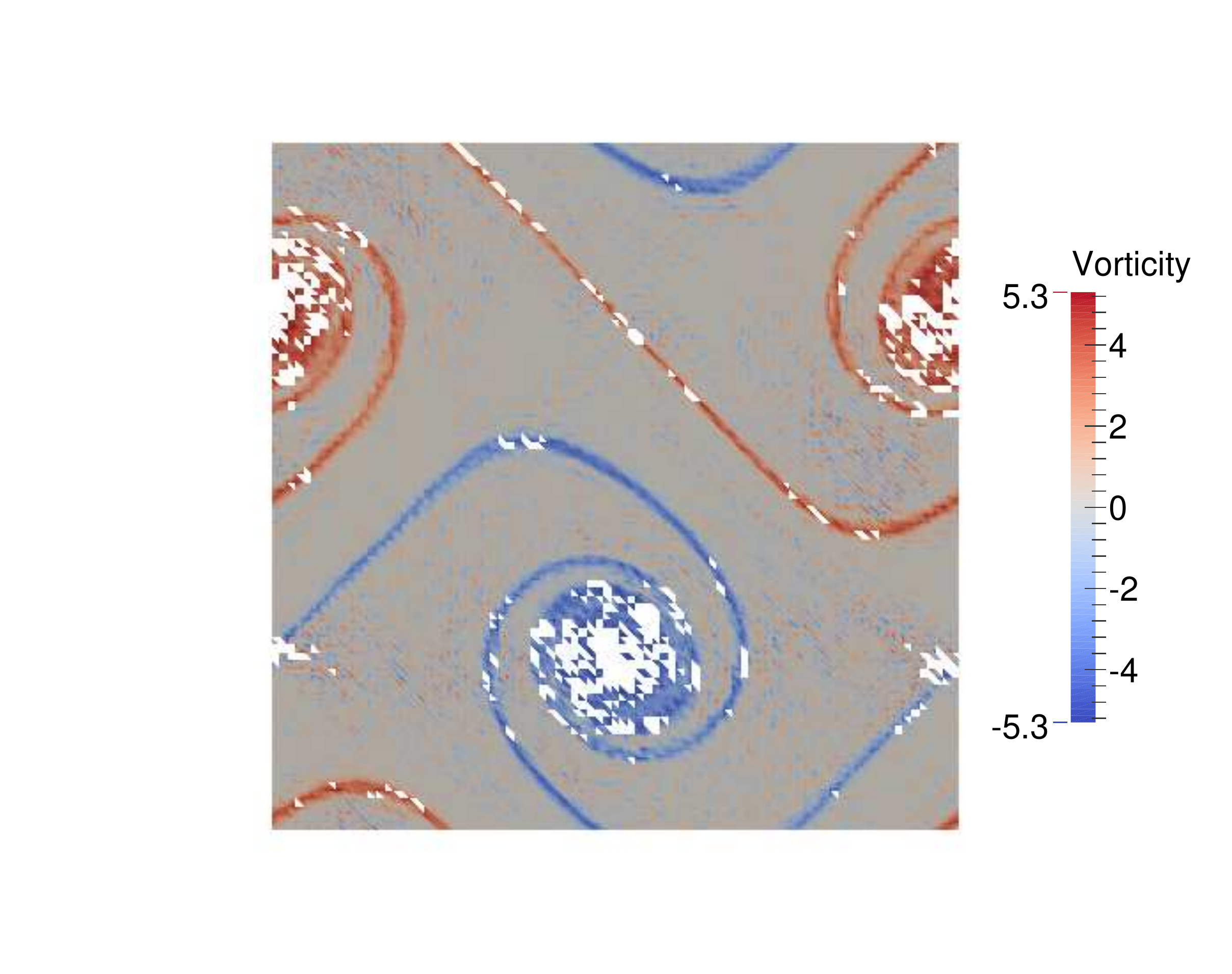}}
\caption{Vorticity function, $\mr{rot}\, {\bf u}$, at $t=8$, centred scheme, ${\bf W}_h = {\bf{BDM}}_s(\mc T_h)$. White colour is used for values outside of range.}
\label{fig:ensnap}
\end{figure}
\begin{figure}
\noindent
\subfigure[$h=0.185$, $s=1$]{\includegraphics[width=0.5\linewidth, trim= 80 50 20 50 ,clip]{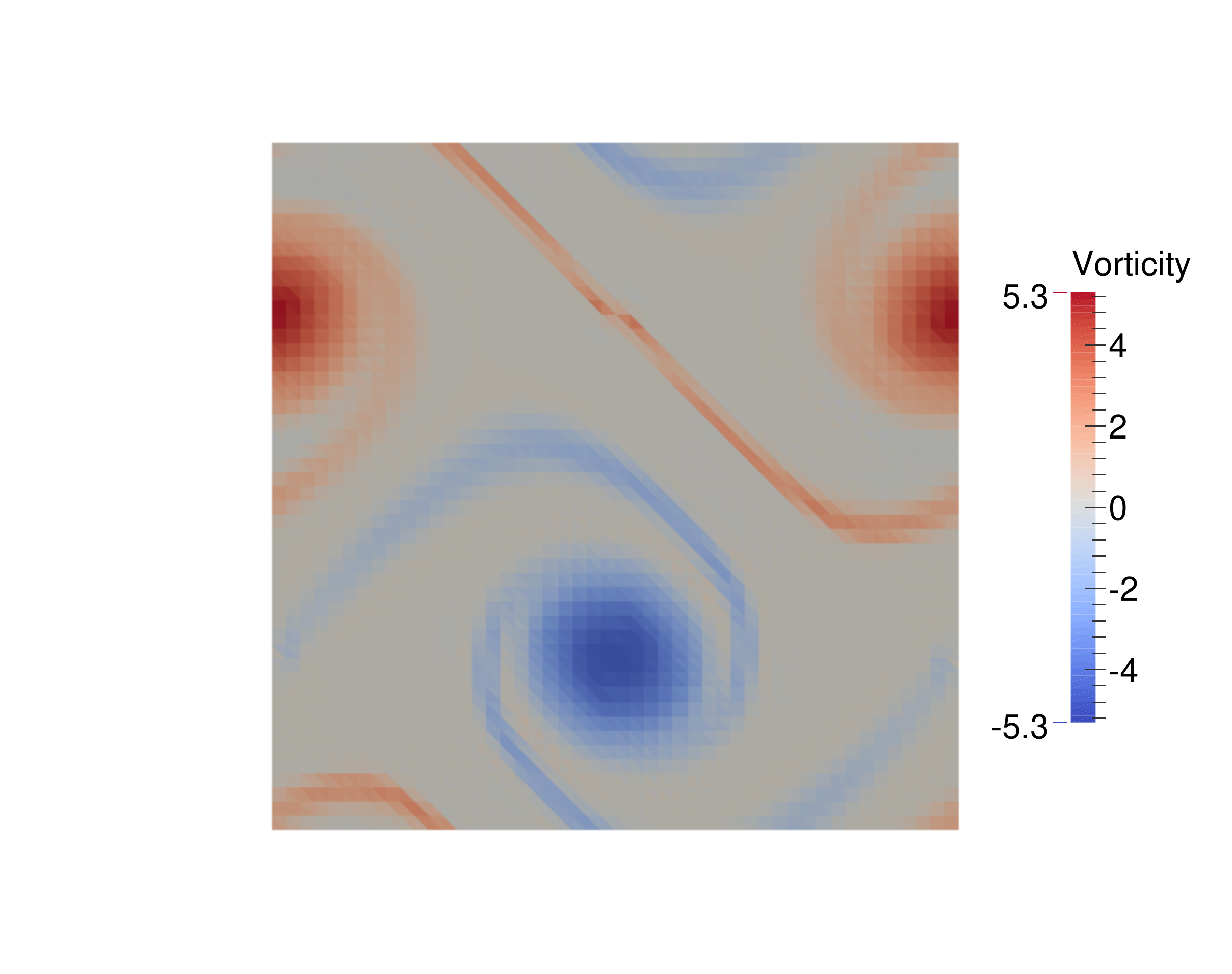}}
\subfigure[$h=0.185$, $s=2$ ]{\includegraphics[width=0.5\linewidth, trim= 80 50 20 50 ,clip]{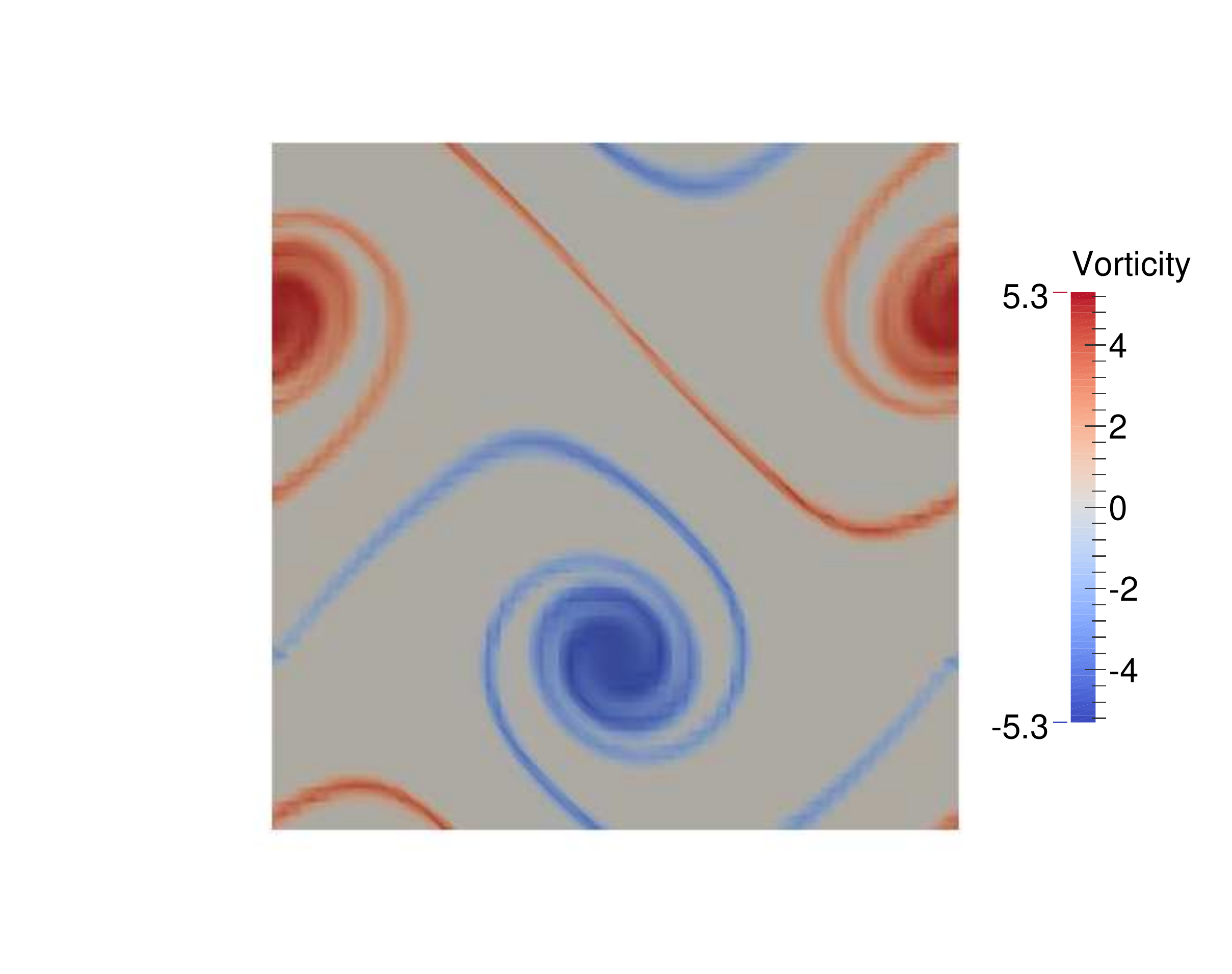}}
\subfigure[$h=0.139$, $s=1$]{\includegraphics[width=0.5\linewidth, trim= 80 50 20 50 ,clip]{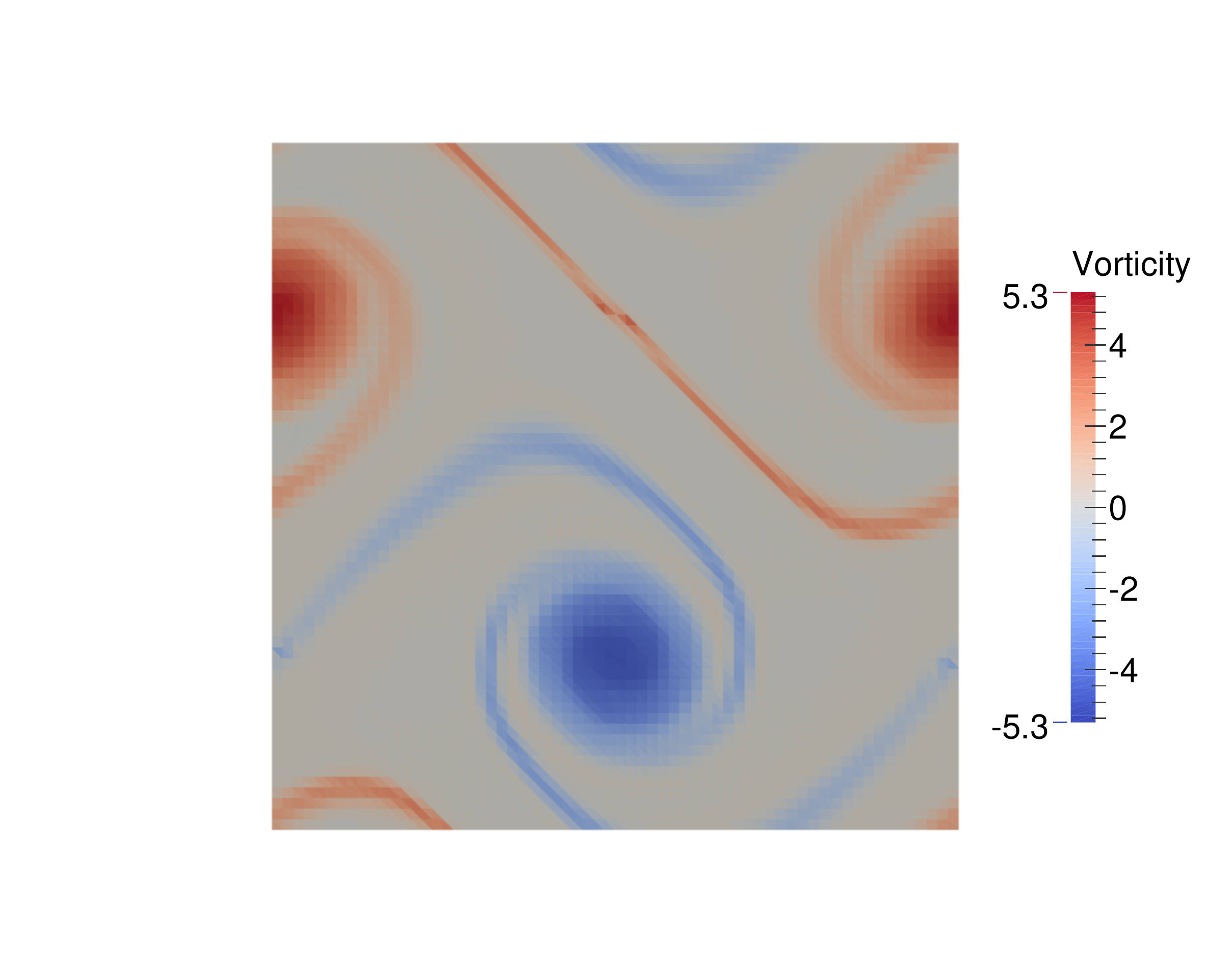}}
\subfigure[$h=0.139$, $s=2$ ]{\includegraphics[width=0.5\linewidth, trim= 80 50 20 50 ,clip]{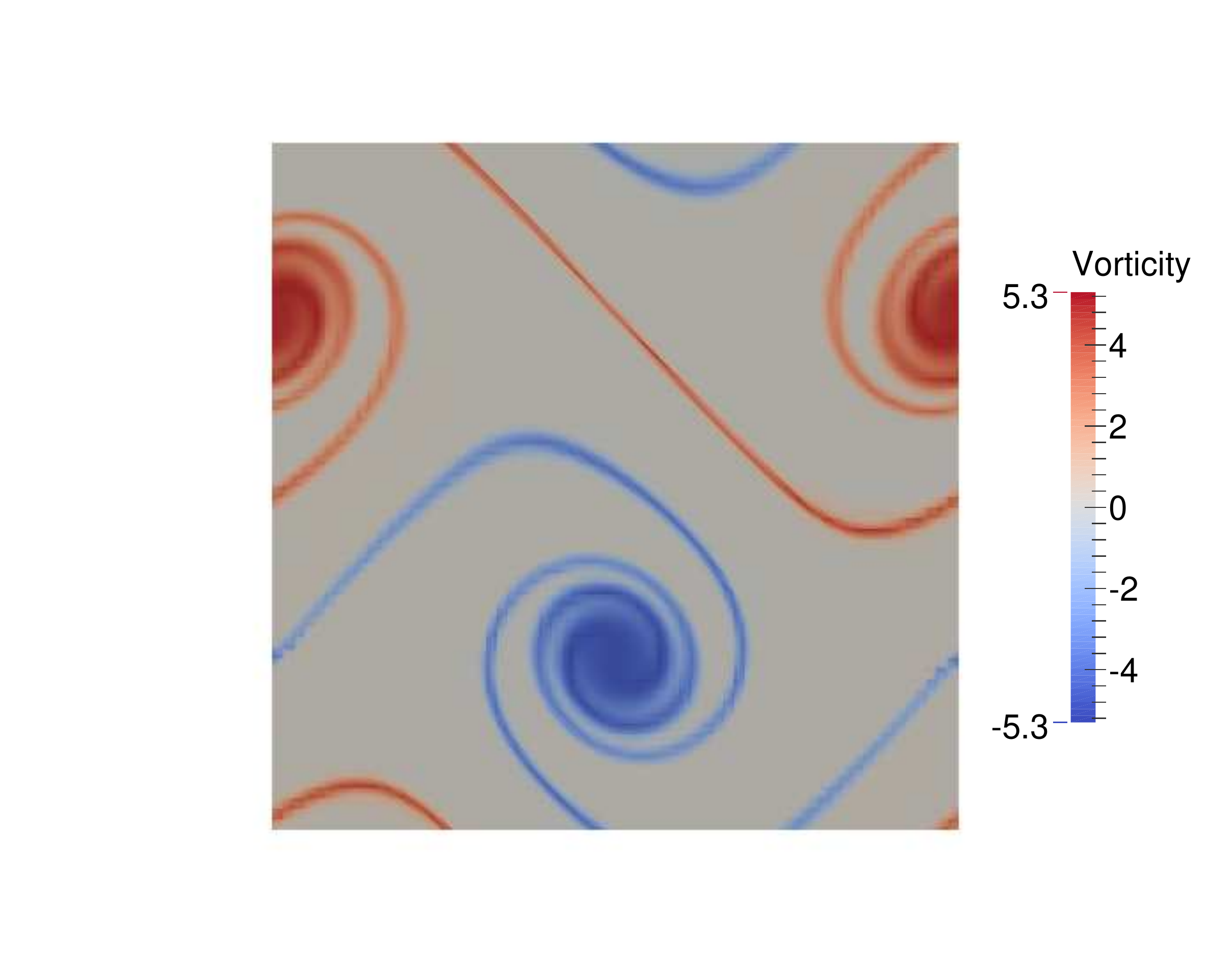}}
\subfigure[$h=0.104$, $s=1$]{\includegraphics[width=0.5\linewidth, trim= 80 50 20 50 ,clip]{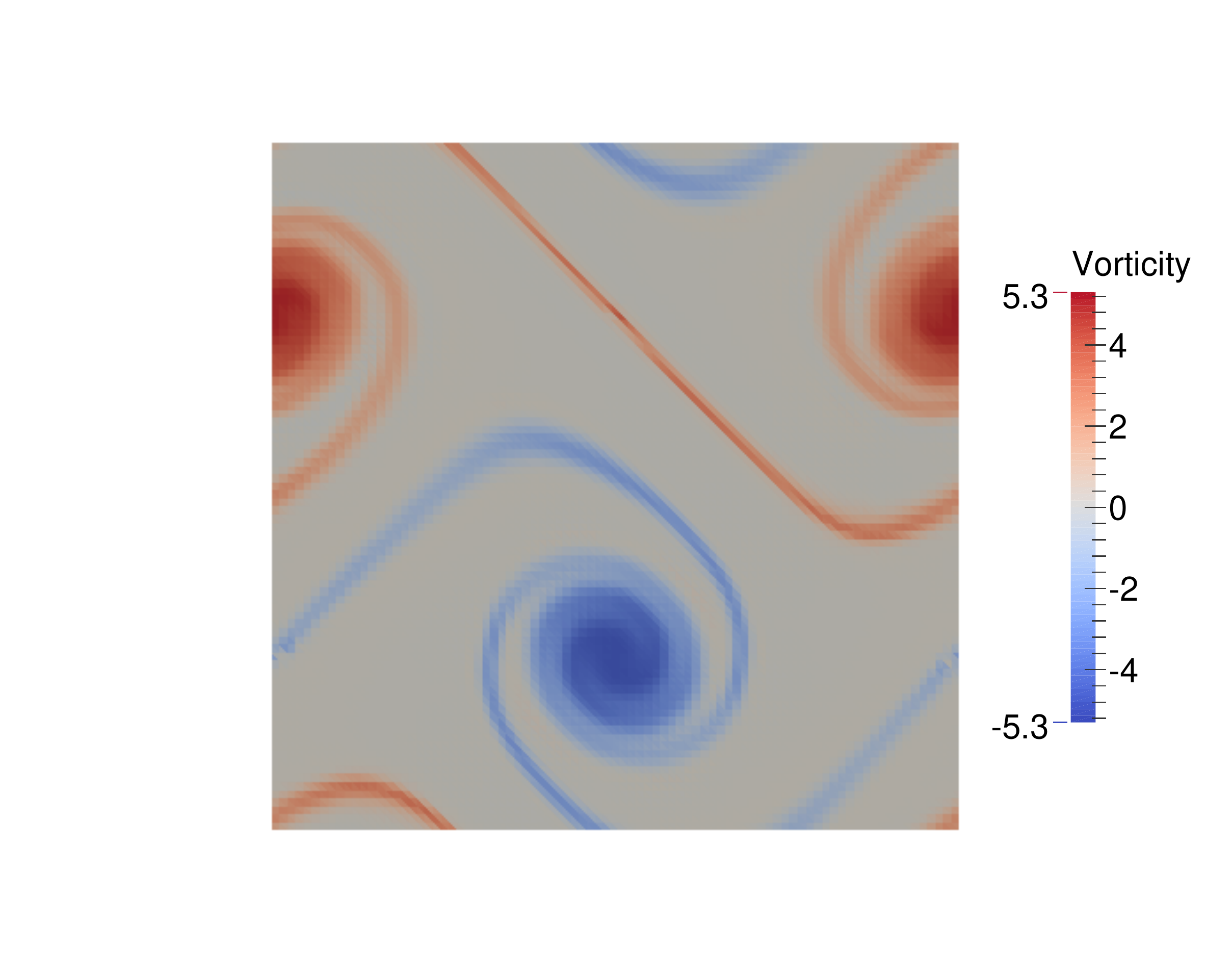}}
\subfigure[$h=0.104$, $s=2$ ]{\includegraphics[width=0.5\linewidth, trim= 80 50 20 50 ,clip]{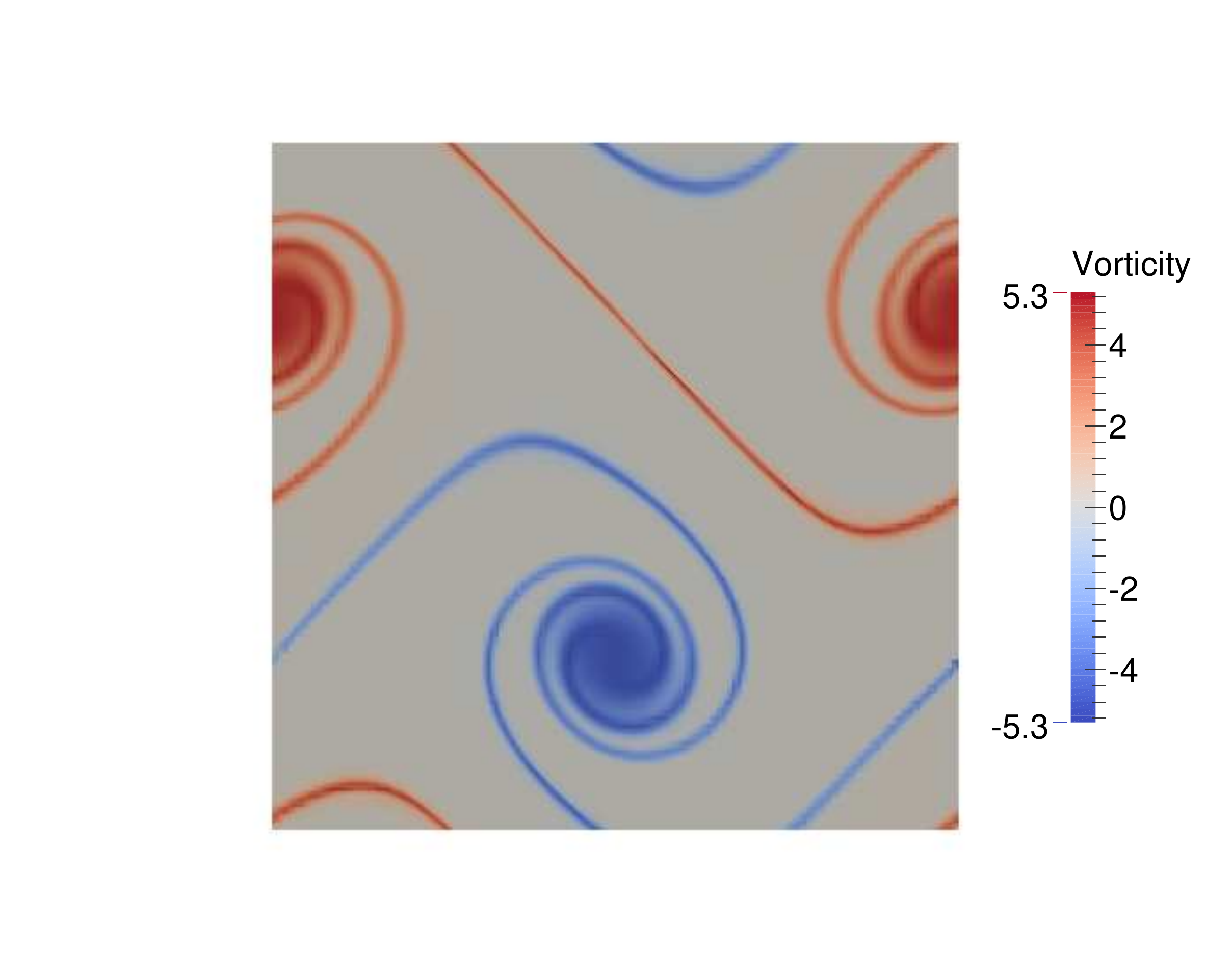}}
\caption{Vorticity function, $\mr{rot}\, {\bf u}$, at $t=8$, upwind Lie derivative scheme, ${\bf W}_h = {\bf{BDM}}_s(\mc T_h)$.}
\label{fig:ensnap1}
\end{figure}
\begin{figure}
\noindent
\subfigure[$h=0.185$, $s=1$]{\includegraphics[width=0.5\linewidth, trim= 80 50 20 50 ,clip]{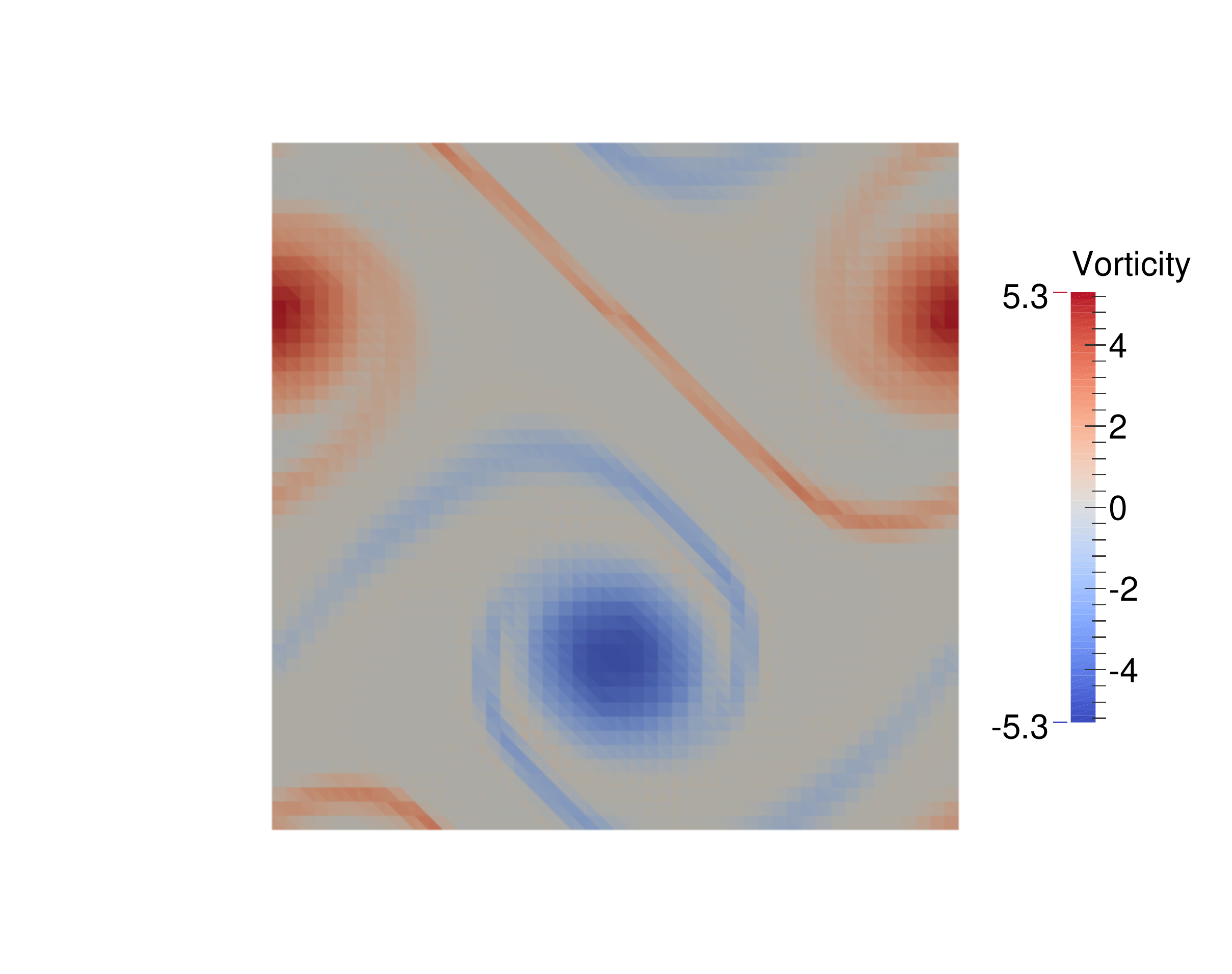}}
\subfigure[$h=0.185$, $s=2$ ]{\includegraphics[width=0.5\linewidth, trim= 80 50 20 50 ,clip]{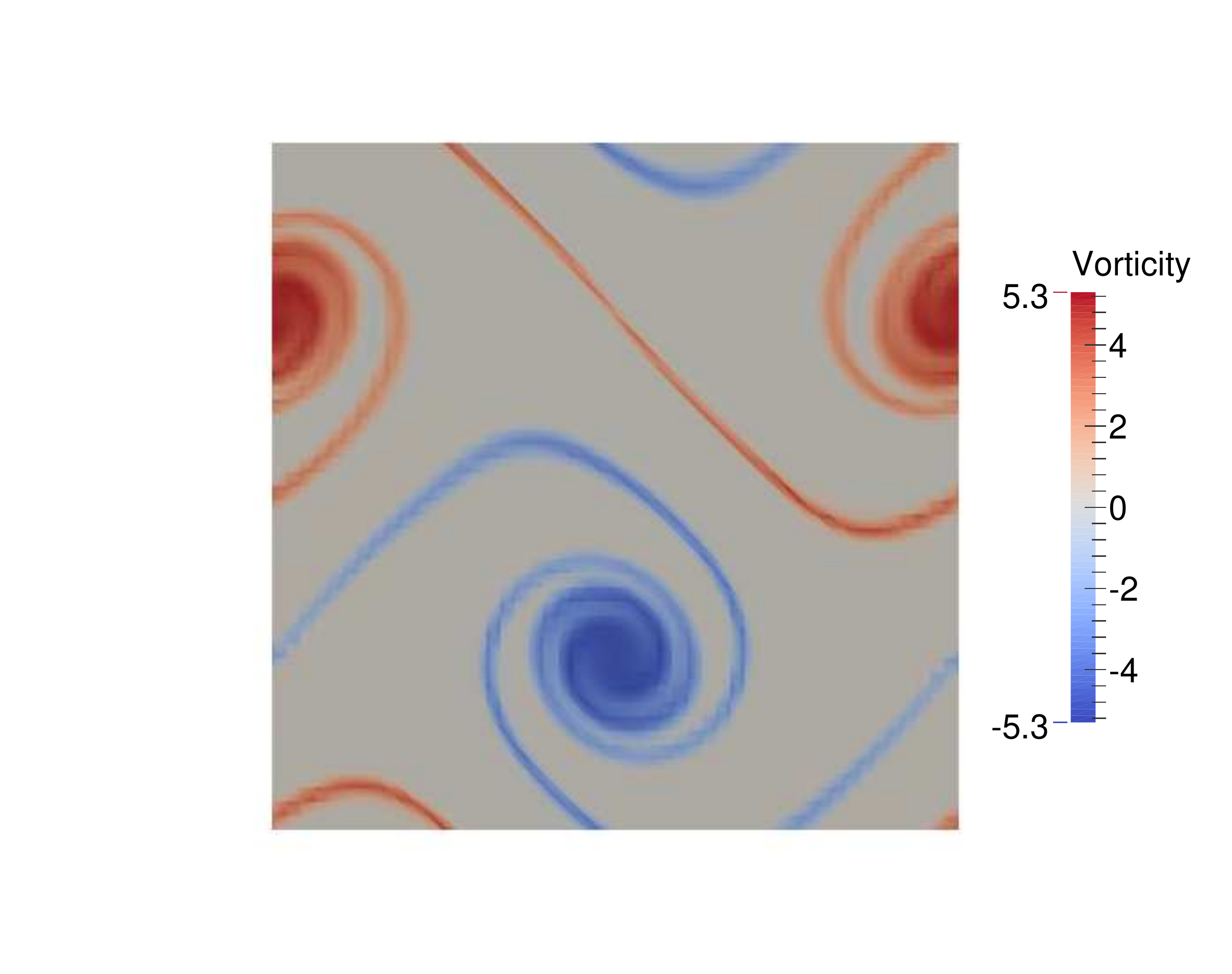}}
\subfigure[$h=0.139$, $s=1$]{\includegraphics[width=0.5\linewidth, trim= 80 50 20 50 ,clip]{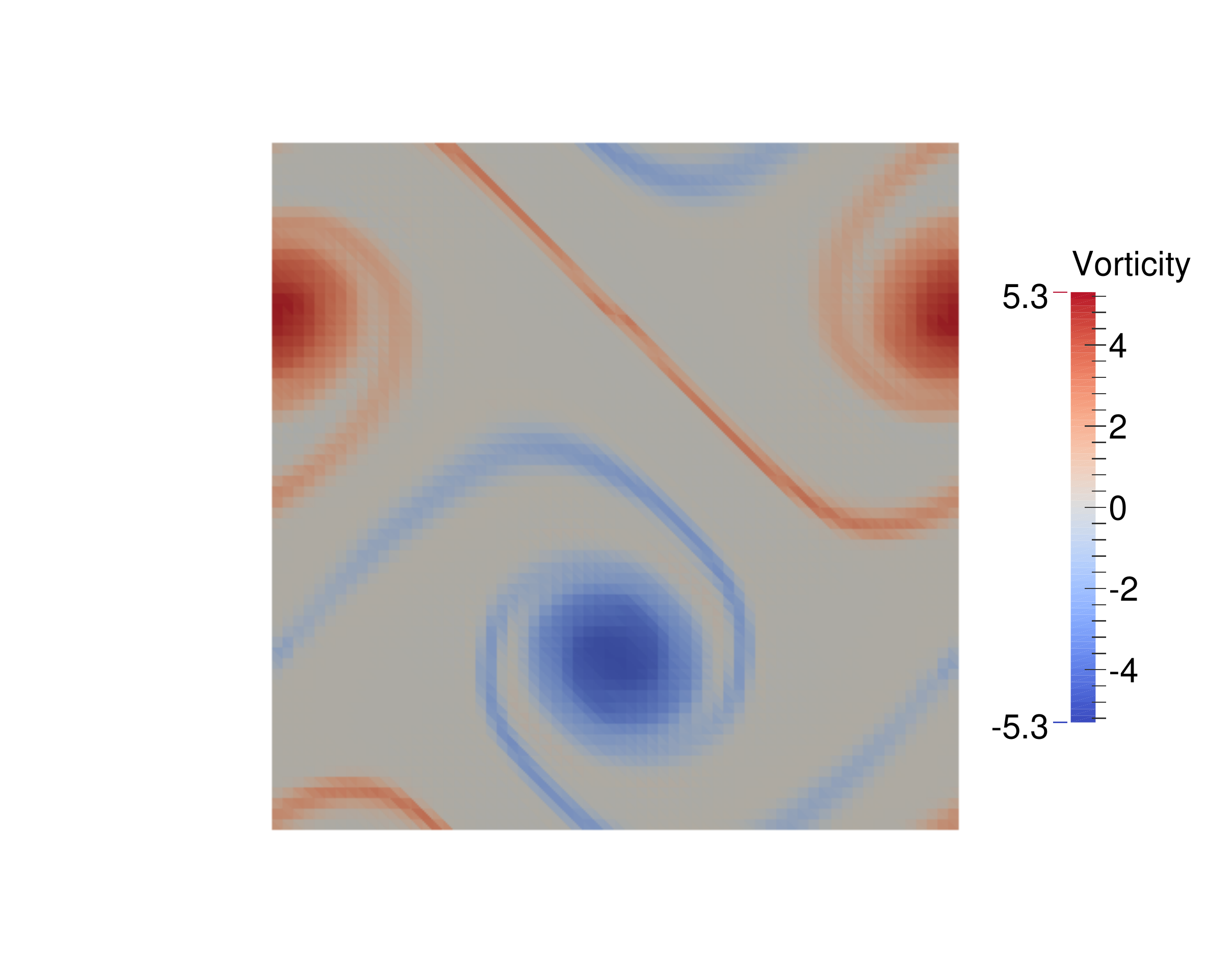}}
\subfigure[$h=0.139$, $s=2$ ]{\includegraphics[width=0.5\linewidth, trim= 80 50 20 50 ,clip]{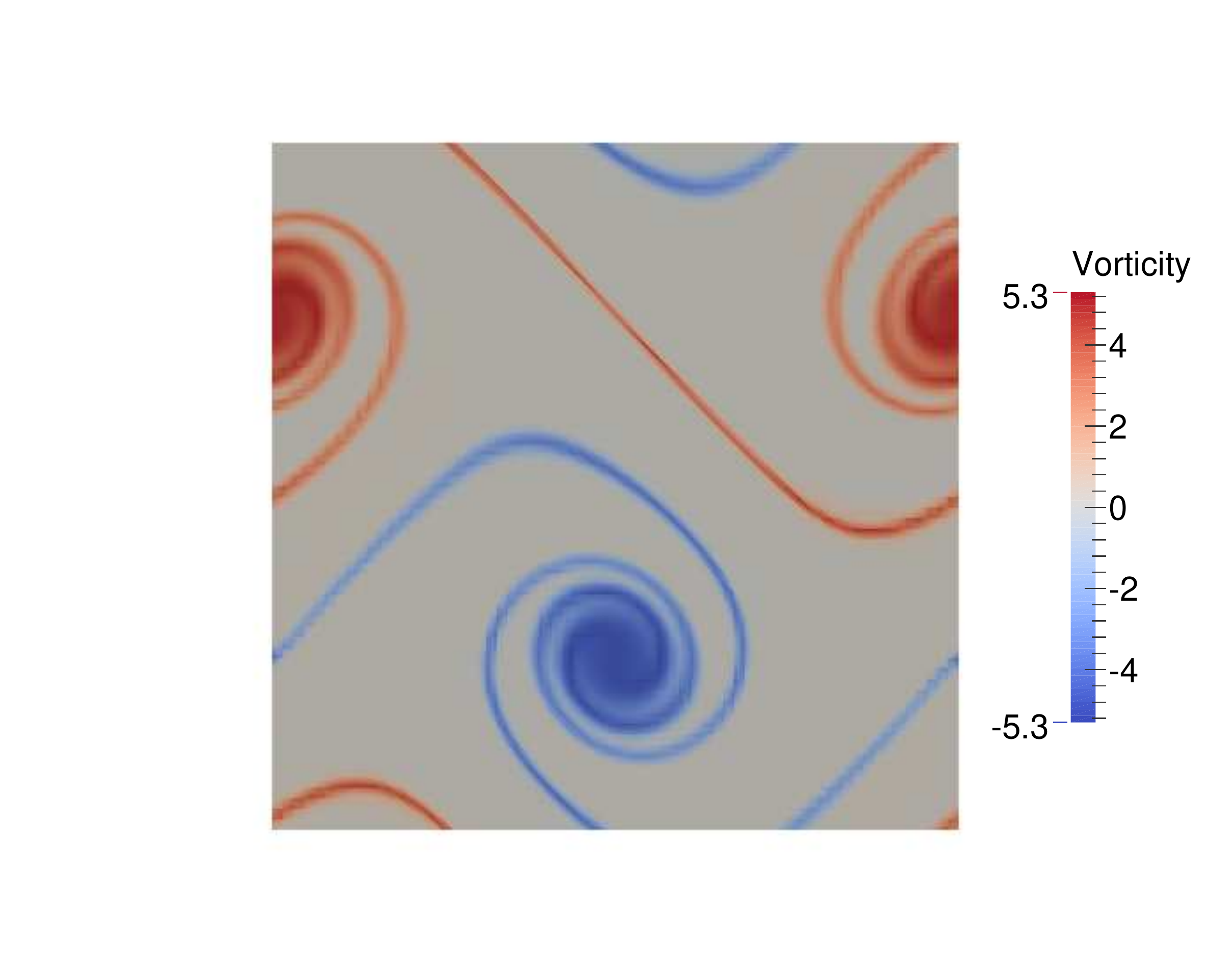}}
\subfigure[$h=0.104$, $s=1$]{\includegraphics[width=0.5\linewidth, trim= 80 50 20 50 ,clip]{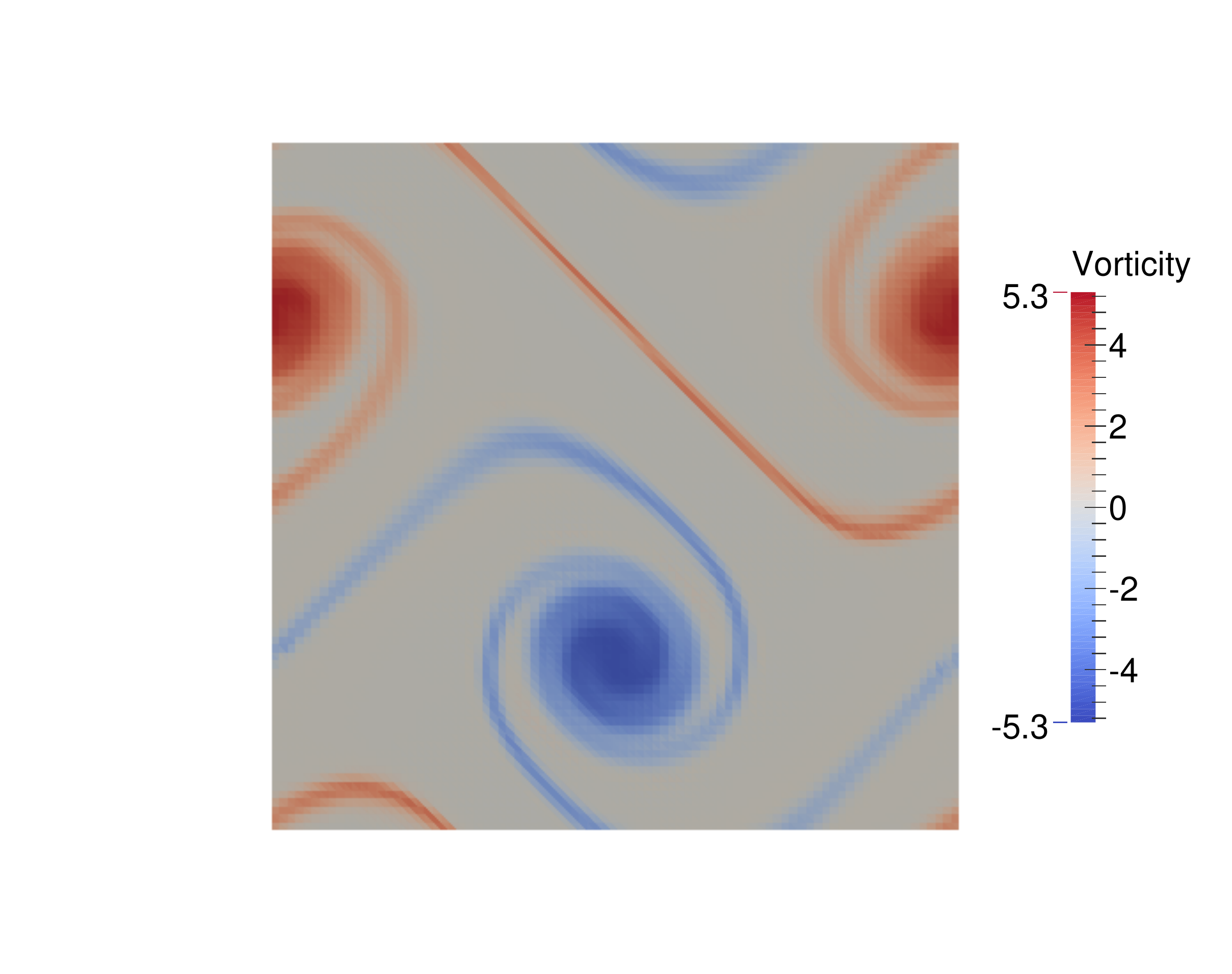}}
\subfigure[$h=0.104$, $s=2$ ]{\includegraphics[width=0.5\linewidth, trim= 80 50 20 50 ,clip]{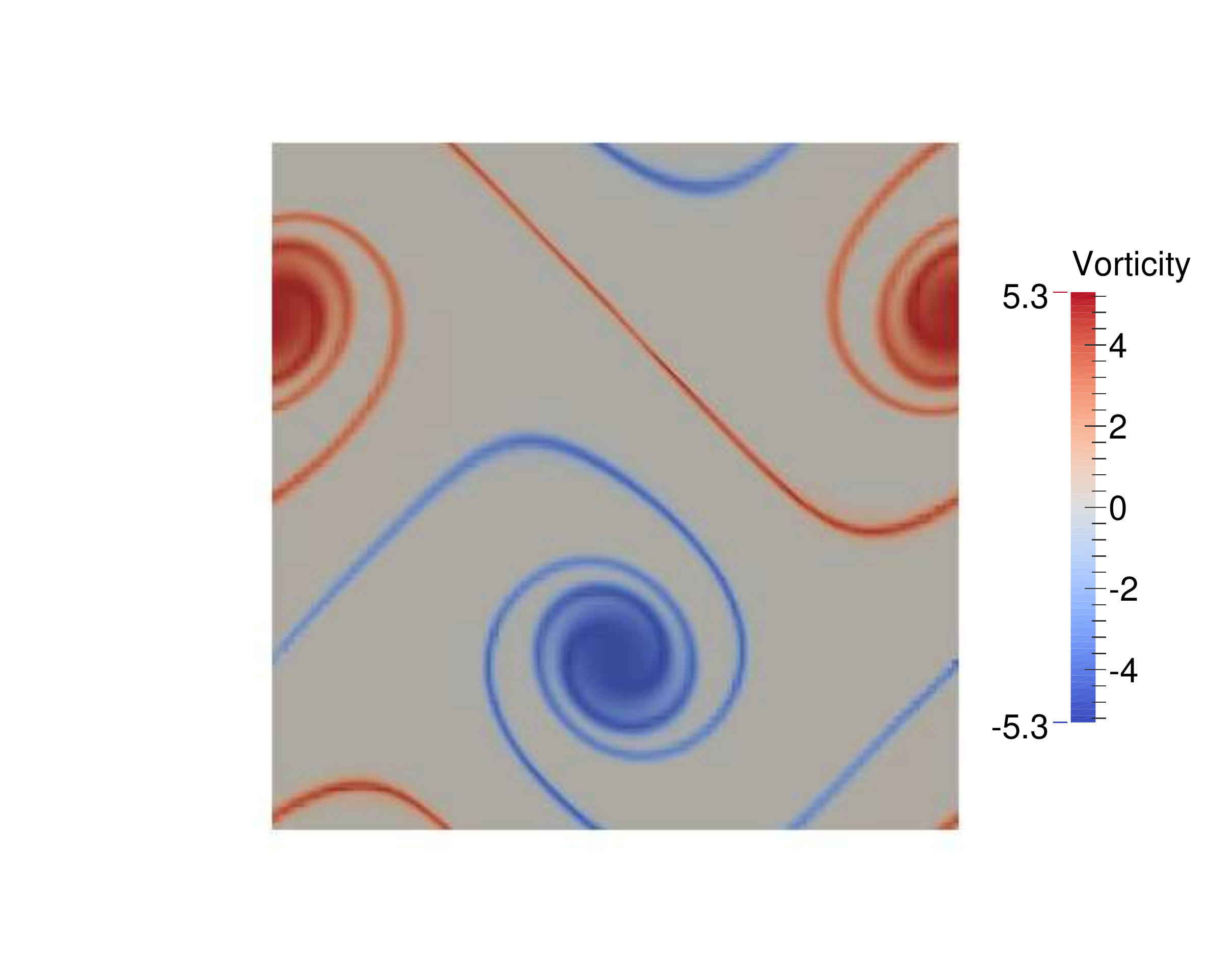}}
\caption{Vorticity function, $\mr{rot}\, {\bf u}$, at $t=8$, upwind GSS scheme, ${\bf W}_h = {\bf{BDM}}_s(\mc T_h)$.}
\label{fig:ensnap2}
\end{figure}
\begin{figure}
\noindent
\subfigure{\includegraphics[width=0.5\linewidth, trim= 0 5 53 30 ,clip]{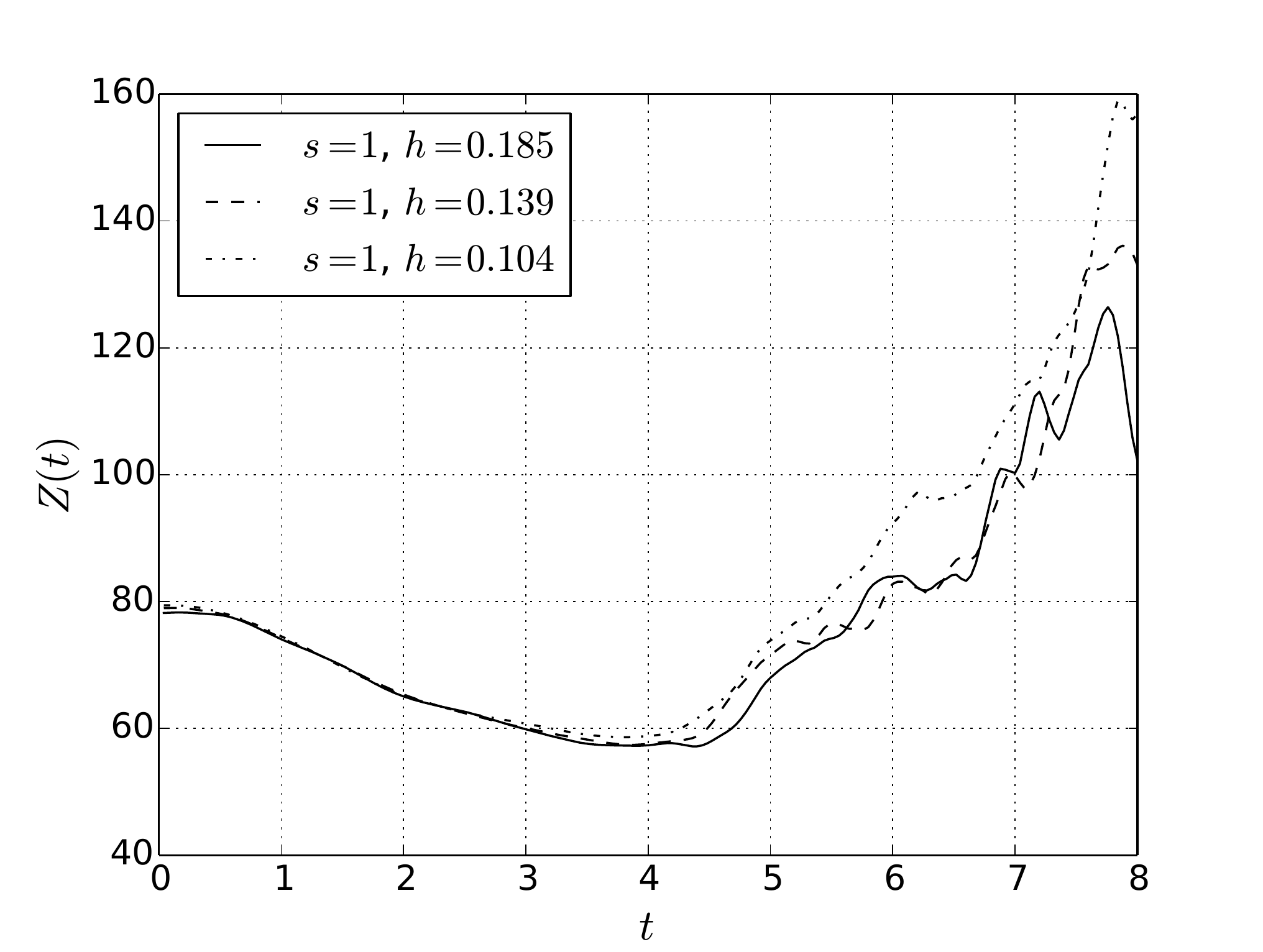}}
\subfigure{\includegraphics[width=0.5\linewidth, trim=  0 5 53 30,clip]{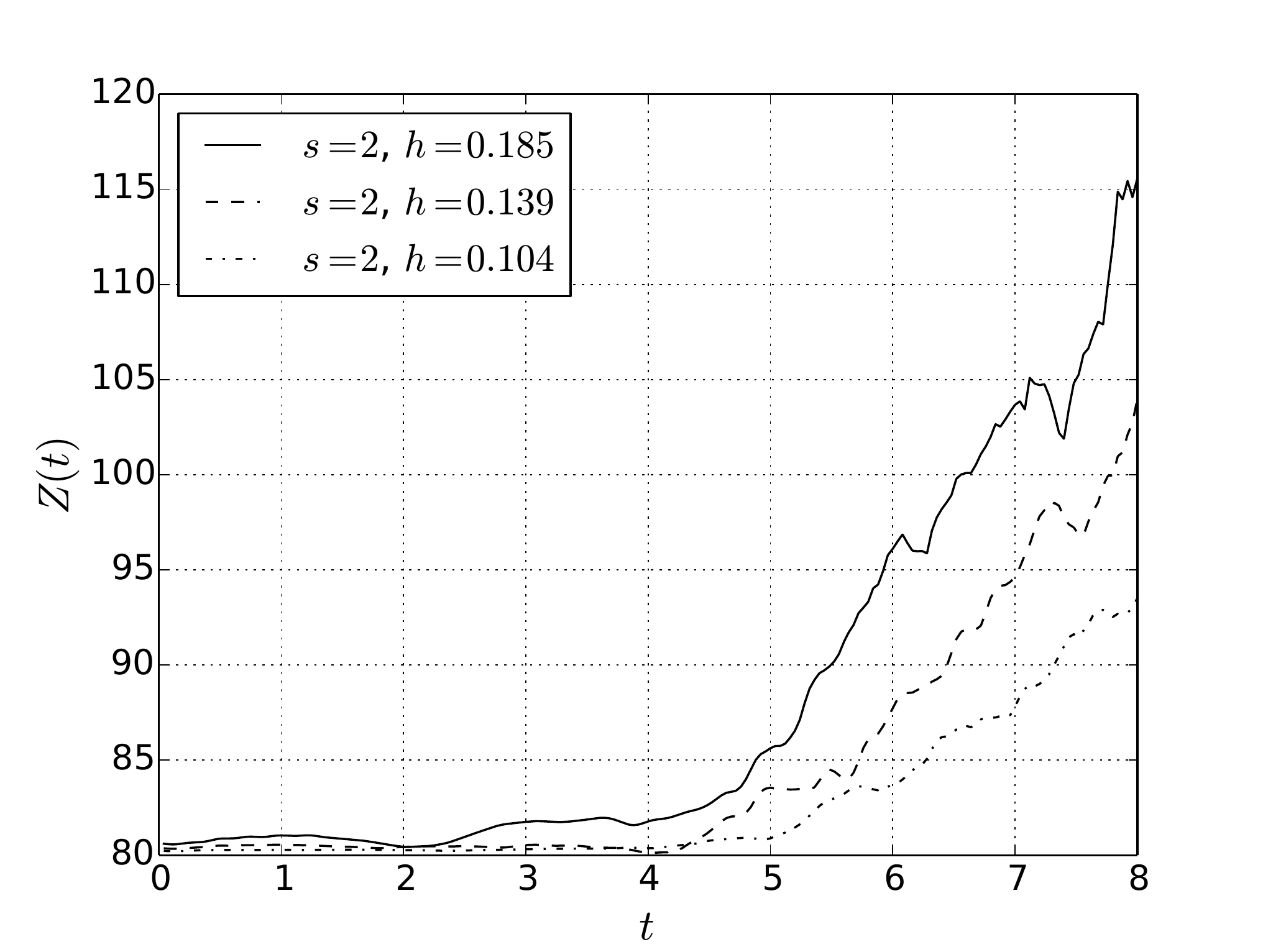}}
\caption{Enstrophy history $Z(t)$ for the centred scheme with ${\bf W}_h = {\bf{BDM}}_s(\mc T_h)$.}
\label{fig:enshist1}
\end{figure}

\begin{figure}
\noindent
\subfigure{\includegraphics[width=0.5\linewidth, trim= 0 5 53 30 ,clip]{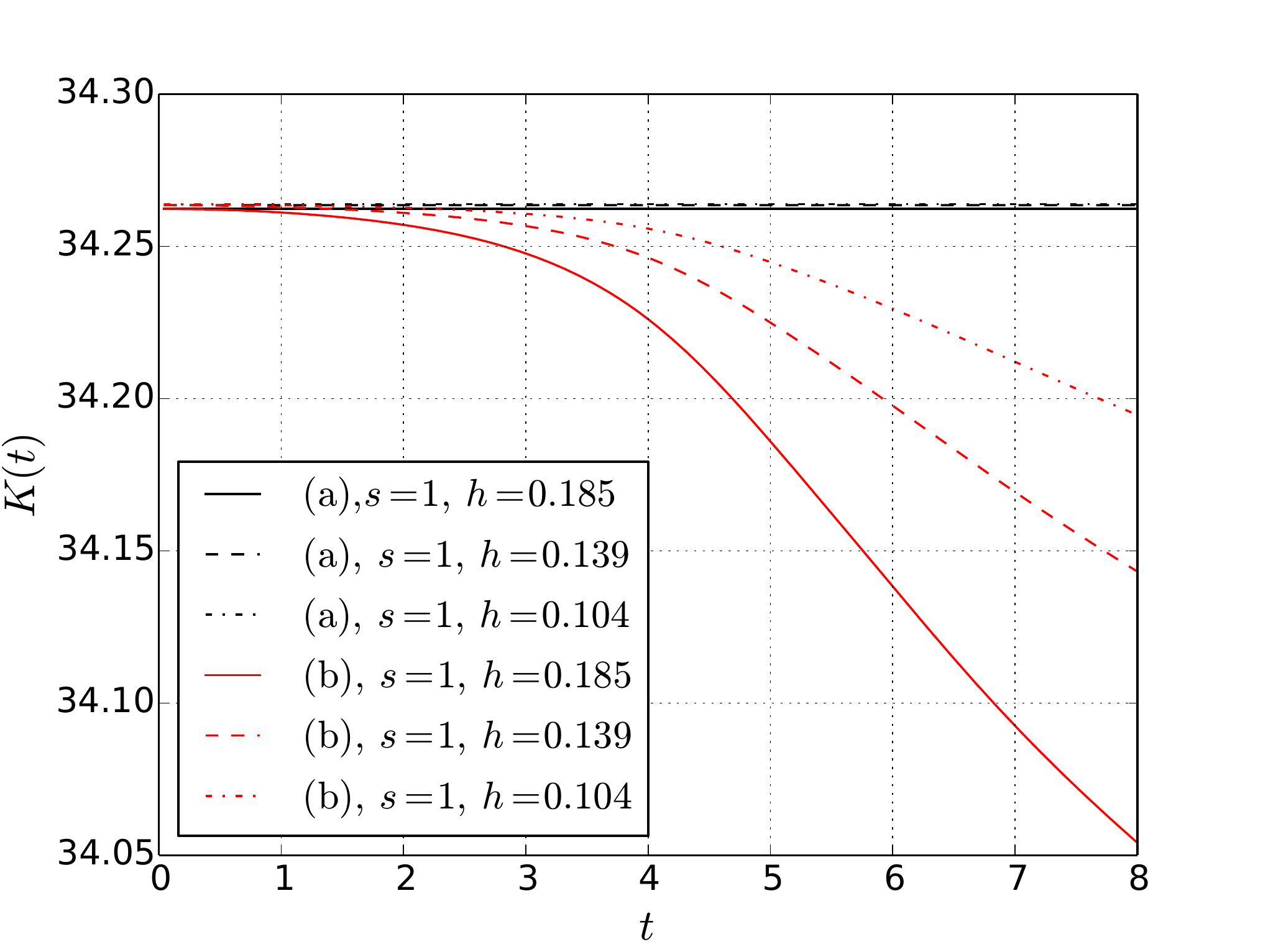}}
\subfigure{\includegraphics[width=0.5\linewidth, trim=  0 5 53 30,clip]{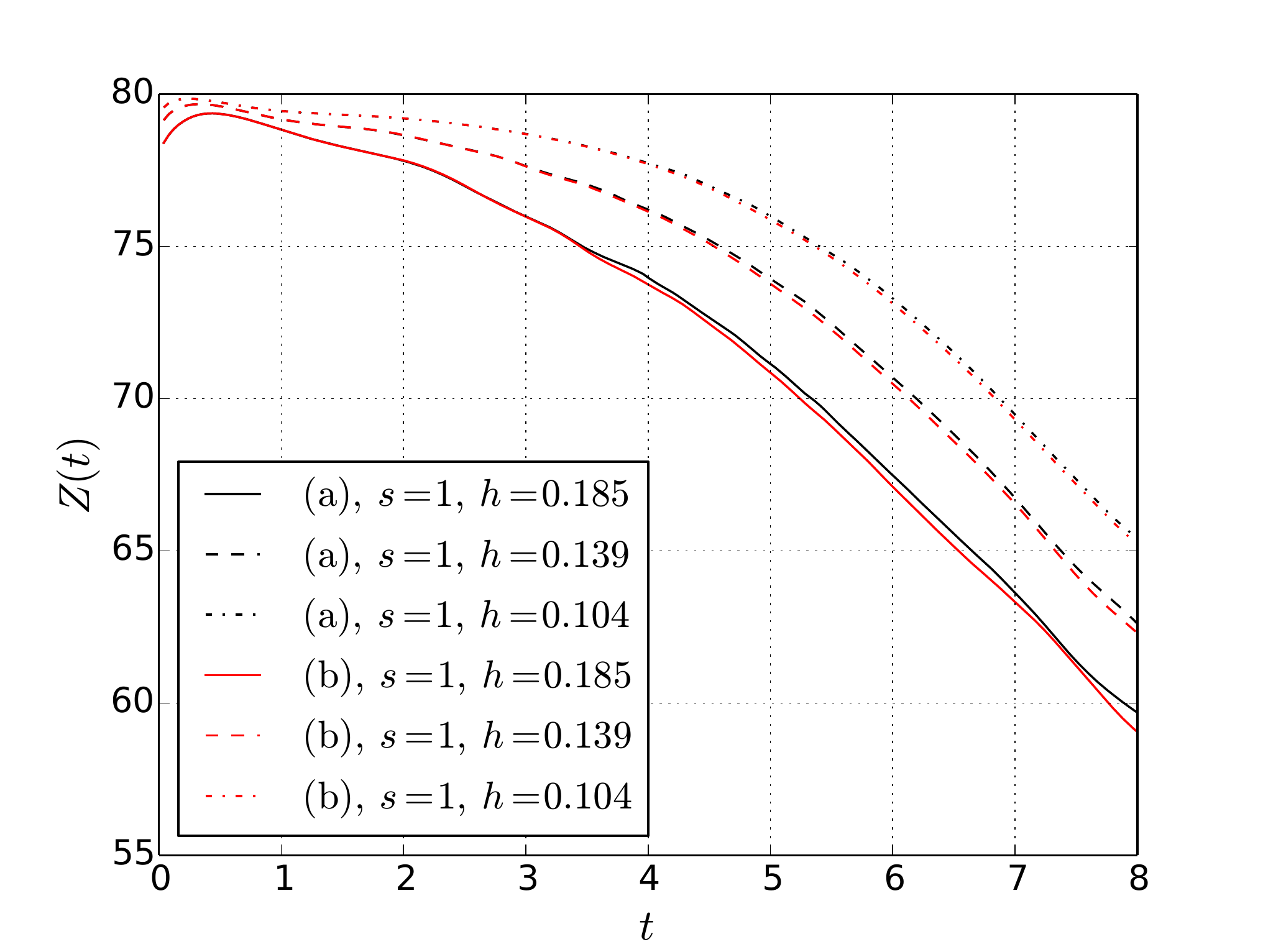}}
\subfigure{\includegraphics[width=0.5\linewidth, trim= 0 5 53 25 ,clip]{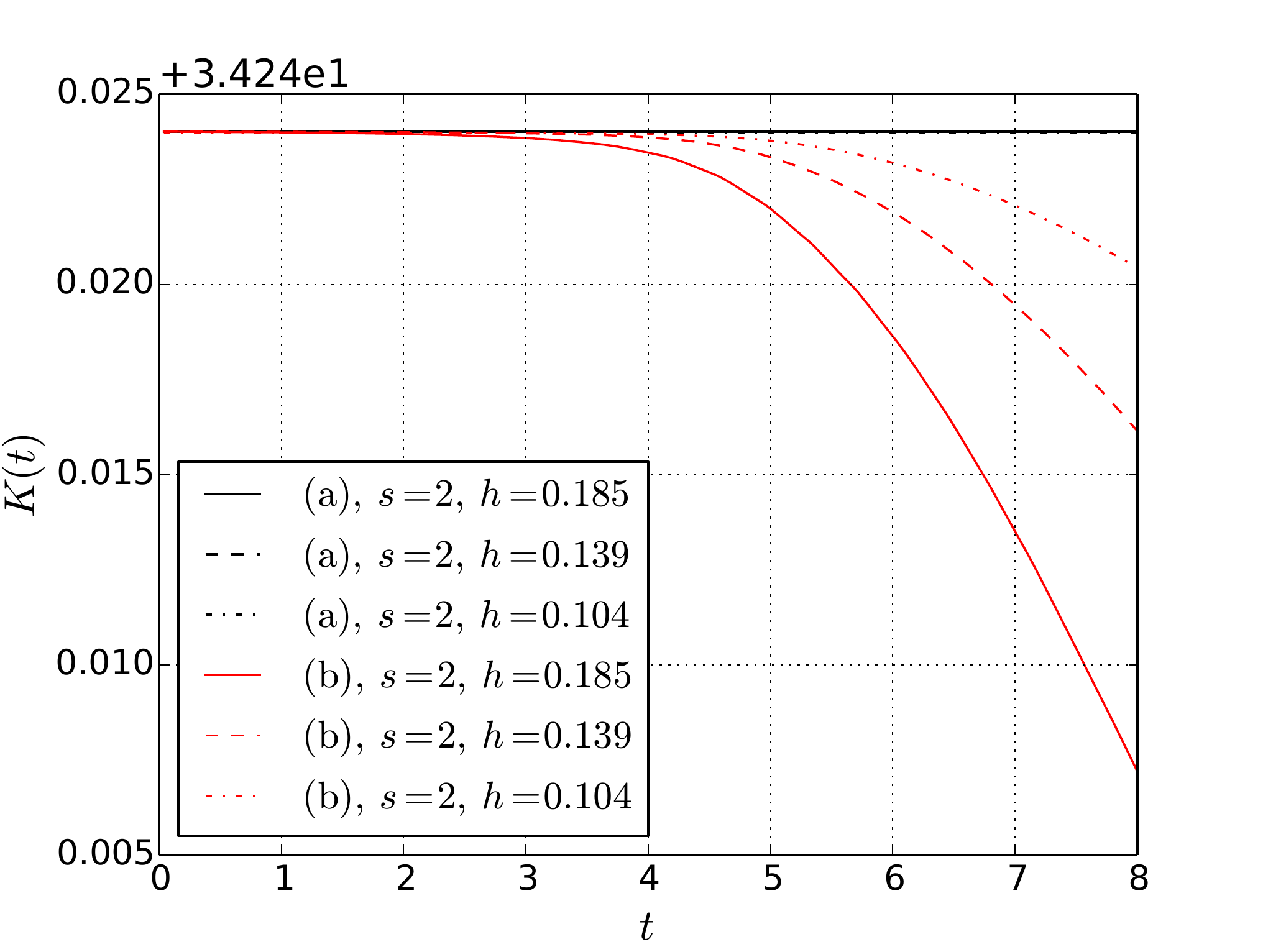}}
\subfigure{\includegraphics[width=0.5\linewidth, trim=  0 5 53 25,clip]{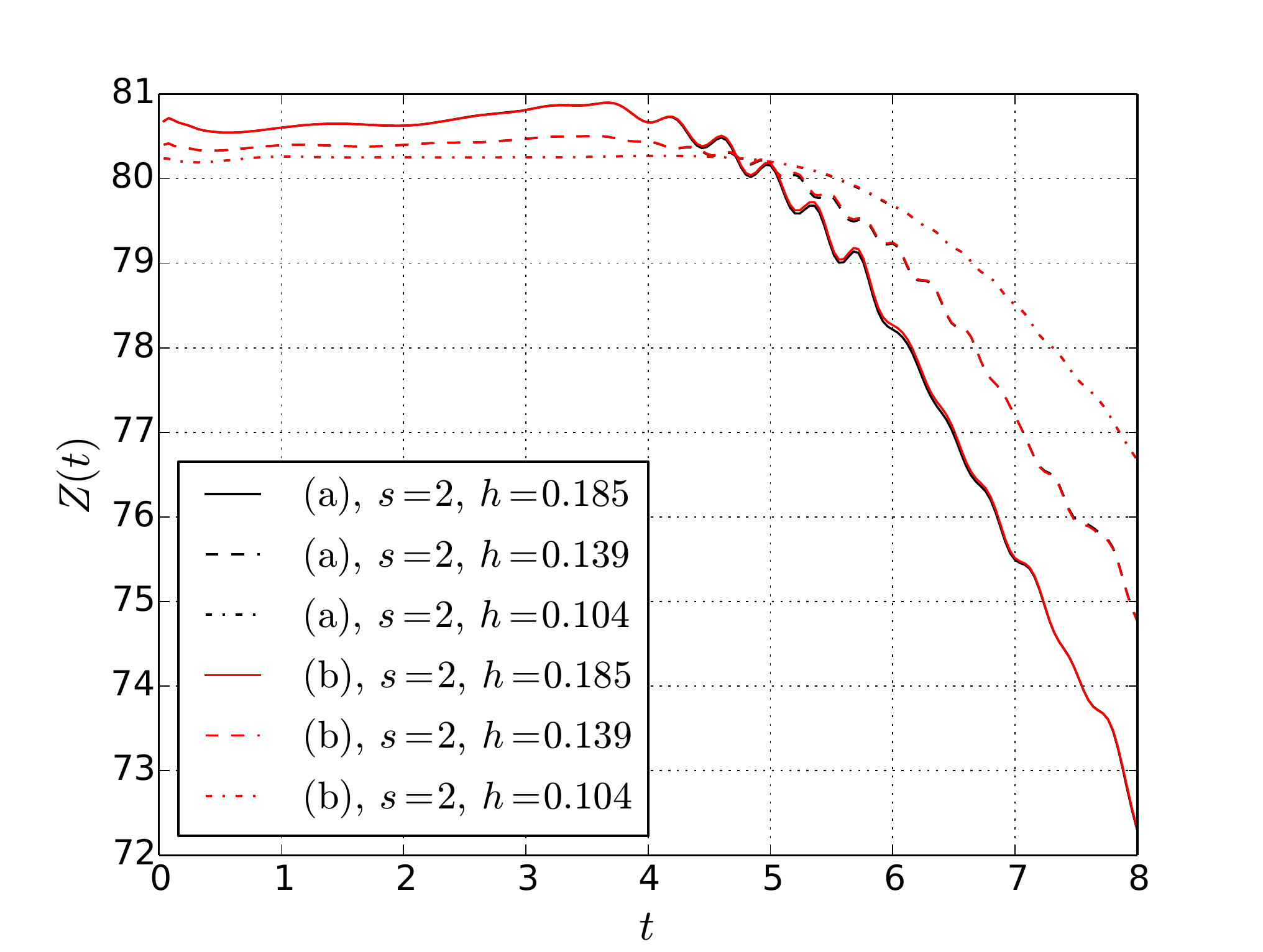}}
\caption{Kinetic energy, $K(t)$, and enstrophy history, $Z(t)$, for the upwind Lie derivative scheme (a), and the upwind GSS scheme (b), with ${\bf W}_h = {\bf{BDM}}_s(\mc T_h)$.}
\label{fig:enshist2}
\end{figure}

\section{Summary and outlook}\label{sec:conc}
In this paper we extended the variational discretisation proposed in \cite{Pavlov1} for the incompressible Euler equations to the finite element setting.  We largely based our efforts on the work of Heumann et al.\ \cite{Heumann11, Heumann13, Heumann16}, where the authors introduced Galerkin discretisations for the Lie derivative operator. Specifically, we used such operators as discrete Lie algebra variables, so that the discrete equations of motion could be derived directly from an appropriately defined Lagrangian and the Hamilton-d'Alembert's principle. We found that the discretisation obtained using this strategy coincides with the centred flux scheme proposed in \cite{Guzman16}. Moreover, it has a built-in energy conservation property and satisfies a discrete Kelvin's circulation theorem. We also introduced an upwind-stabilised version of the algorithm, which preserves both properties.
Finally, we provided a convergence analysis proving (sub-optimal) convergence rates for the upwind scheme for sufficiently high order finite elements. Numerical tests suggest that this result might not be sharp, as we obtain optimal convergence rates even using low order spaces.

The methodology that we used in this paper to discretise the incompressible Euler equations is general and can be  applied to different fluid models which share a similar variational structure. The simplest example of such possible extensions is given by the Euler-alpha model \cite{Holm98}. In this case, we keep the setting of the Euler equations described in Section~\ref{sec:models}, whereas the Lagrangian is given by
\[
l({\bf u}) = \frac{1}{2} \int_{\Omega } (\|{\bf u} \|^2 + \alpha^2 \|\bf{grad}\,{\bf u}\|^2\,)\, \mr{vol} \, ,
\]
and $\delta{l}/\delta{\bf u}$ can be identified with an opportunely defined momentum $ m = A {\bf u} \cdot \ed{\bf x}$, with $A = (I-\alpha^2\Delta)$ and $I$ being the identity map. Suppose that we have a discretisation for $A$ given by $A_h$. Then, in the notation of Section~\ref{sec:fevari}, we can take $A_h:{\bf W}_h \rightarrow {\bf W}_h$ to be a linear invertible operator acting on the velocity finite element space ${\bf W}_h$. If we assume $A_h$ to be symmetric, the variational derivation we discussed in Section~\ref{sec:fevari} can be applied without changes so that eventually we get an advection equation for the momentum $m$ expressed in terms of our discrete Lie derivative operator.

Many MHD and GFD models can also be derived from a variational principle similar to the one that applies for perfect incompressible fluids \cite{Holm98}. In these cases one needs to add advected quantities to the system, which again require an appropriate definition for a discrete Lie derivative. Moreover, the variational derivation of Section~\ref{sec:fevari}, needs to be appropriately modified, although we expect to maintain its general features. This provides a direction for our future research.

\bibliography{thesis1}
\bibliographystyle{plain}
 
\end{document}